\documentclass[11pt, reqno]{amsart}
\usepackage[a4paper, hmargin={2.7cm,2.7cm},vmargin={3.3cm,3.3cm}]{geometry}
\usepackage[nospace,noadjust]{cite}
\usepackage{amsmath}
\usepackage{amssymb}
\usepackage{amsfonts}
\usepackage{amsthm}
\usepackage{enumerate}
\usepackage[T1]{fontenc}
\usepackage{dsfont}

\usepackage{etoolbox}

\makeatletter
\patchcmd{\ttlh@hang}{\parindent\z@}{\parindent\z@\leavevmode}{}{}
\patchcmd{\ttlh@hang}{\noindent}{}{}{}
\makeatother

\newcommand\numberthis{\addtocounter{equation}{1}\tag{\theequation}}

\newtheorem{theorem}{Theorem}[section]
\newtheorem{lemma}[theorem]{Lemma}
\newtheorem{proposition}[theorem]{Proposition}
\newtheorem{corollary}[theorem]{Corollary}

\theoremstyle{definition}
\newtheorem{definition}[theorem]{Definition}

\newtheorem{examplex}[theorem]{Example}
\newenvironment{example}
  {\pushQED{\qed}\examplex}
  {\popQED\endexamplex}
 
\theoremstyle{remark}
\newtheorem{remark}[theorem]{Remark}

\numberwithin{equation}{section}

\makeatletter
\def\paragraph{\@startsection{paragraph}{4}%
  \z@\z@{-\fontdimen2\font}%
  {\normalfont\bfseries}}
\makeatother

\DeclareMathOperator*{\Co}{Co}

\DeclareMathOperator*{\loc}{loc}

\DeclareMathOperator*{\supp}{supp}

\DeclareMathOperator*{\esssup}{ess\,sup}

\newcommand{\SOdual}{\overline{\mathcal{S}}_{\mathcal{O}}'(\RR^d)}
\newcommand{\Schwartz}{\mathcal{S}}

\newcommand{\RR}{\mathbb{R}}
\newcommand{\CC}{\mathbb{C}}
\newcommand{\RHat}{\widehat{\mathbb{R}}}
\newcommand{\CoLw}{\Co (L^{p,q}_v(G))}
\newcommand{\CoSY}{\Co (Y)}

\newcommand{\CoHLw}{\sideset{}{_\mathcal{H}} \Co (L^{p,q}_v (G))}
\newcommand{\CoSYdual}{\Co (Y')}
\newcommand{\CoSYpsi}{\sideset{}{_{\psi}} \Co (Y)}
\newcommand{\CoSYone}{\sideset{}{_{\psi_1}} \Co (Y)}
\newcommand{\CoSYtwo}{\sideset{}{_{\psi_2}} \Co (Y)}

\author{Hartmut F\"uhr} 
\address{Lehrstuhl A f\"ur Mathematik, RWTH Aachen University, D-52056 Aachen, Germany}
\email{fuehr@matha.rwth-aachen.de}

\author{Jordy Timo van Velthoven}
\address{Faculty of Mathematics, University of Vienna, Oskar-Morgenstern-Platz 1, A-1090 Vienna, Austria.}
\email{jordy-timo.van-velthoven@univie.ac.at}

\thanks{
The authors thank the referee for his/her careful reading and the thoughtful comments.
J.~v.~V. acknowledges support from the Austrian Science Fund (FWF) P 29462 - N35 and Y-1199.}


\title{Coorbit spaces associated to integrably admissible dilation groups}
\date{}

\begin{document}

\maketitle

\begin{abstract}
This paper considers coorbit spaces parametrized by mixed, weighted Lebesgue spaces
with respect to the quasi-regular representation of the semi-direct product 
of Euclidean space and a suitable matrix dilation group. 
The class of dilation groups that we allow, the so-called
integrably admissible dilation groups, 
 contains the matrix groups 
yielding an irreducible, square-integrable quasi-regular representation as a proper subclass.
The obtained scale of coorbit spaces extends therefore the well-studied wavelet coorbit spaces 
associated to discrete series representations. 
We show that for any integrably admissible dilation group there exists
a convienent space of smooth, admissible analyzing vectors
that can be used to define a consistent coorbit space possessing all the essential 
properties that are known to hold in the setting of discrete series representations.  
In particular, the obtained coorbit spaces can be realized
as Besov-type decomposition spaces by means of a Littlewood-Paley type characterization. 
The classes of anisotropic
Besov spaces associated to expansive matrices are shown to coincide precisely with the coorbit spaces 
induced by the integrably admissible one-parameter groups. 
\end{abstract}

\section{Introduction}
Let $H \leq \mathrm{GL}(d, \RR)$ be a closed subgroup and consider the semi-direct product
$G = \RR^d \rtimes H$. The group $G$ acts unitarily on $L^2 (\RR^d)$ through the 
quasi-regular representation $(\pi, L^2 (\RR^d))$ by
\[
\pi (x,h) f = |\det (h) |^{-\frac{1}{2}} f(h^{-1} (\cdot -x)).
\]
Given a non-zero $\psi \in L^2 (\RR^d)$, the associated \emph{wavelet transform} 
$W_{\psi} : L^2 (\RR^d) \to C_b (G)$ 
is defined by the representation coefficients
\[
W_{\psi} f (x,h) = \langle f, \pi (x,h) \psi \rangle. 
\]
A vector $\psi \in L^2 (\RR^d)$ is called \emph{admissible} if $W_{\psi} : L^2 (\RR^d) \to L^2(G)$ 
is an isometry. 
An important consequence of admissibility of 
$\psi \in L^2 (\RR^d)$ is the \emph{reproducing formula}
\begin{align} \label{eq:reproducing_intro}
W_{\psi} f = W_{\psi} f \ast W_{\psi} \psi
\end{align} 
valid for any $f \in L^2 (\RR^d)$. 
The existence of admissible vectors and associated groups has been studied in numerous papers,
including 
\cite{fuehr2002continuous, bernier1996wavelets, fuehr2010generalized, laugesen2002characterization}.
See also \cite{currey2007admissible, oussa2013admissible, liu1997admissible} for similar investigations, 
in which $\RR^d$ is replaced by a non-commutative nilpotent Lie group $N$.  

\subsection{Integrable wavelet transforms}
This paper focuses on wavelet transforms associated with so-called
\emph{integrably admissible dilation groups}; 
see Section \ref{sec:integrably_admissible} for the precise definition. 
An important property of such a dilation group $H \leq \mathrm{GL}(d, \RR)$ is that 
the quasi-regular representation $(\pi, L^2 (\RR^d))$ of $G = \RR^d \rtimes H$
is not merely square-integrable, but also integrable in the sense that there exists an admissible 
$\psi \in L^2 (\RR^d)$ such that $W_{\psi} \psi \in L^1_v (G)$ for some suitable weighting function
$v$ on $G$. 
The integrability of $(\pi, L^2 (\RR^d))$ is a property that is of interest for at least two reasons.
On the one hand, it is exploited in abstract harmonic analysis 
for constructing
projections in the convolution algebra $L^1 (G)$. 
Indeed, the reproducing formula \eqref{eq:reproducing_intro} yields 
that $W_{\psi} \psi = (W_{\psi} \psi)^{\sharp} = W_{\psi} \psi \ast W_{\psi} \psi$, 
with $(W_{\psi} \psi)^{\sharp} (x) = \overline{W_{\psi} \psi (x^{-1})}$. 
Thus, if $F := \Delta_G^{-1/2} W_{\psi} \psi \in L^1 (G)$, then 
$F = F^* = F \ast F$, where $F^* (x) = \Delta_G (x)^{-1} \overline{F(x^{-1})}$ 
denotes the usual involution in $L^1 (G)$,
showing that $F$ is a projection in $L^1 (G)$. 
The existence of projections in $L^1 (G)$ is related to the existence of compact open sets
in the unitary dual $\widehat{G}$, equipped with the so-called Fell topology, and is
an area of ongoing research \cite{currey2016integrable, kaniuth1996minimal, grochenig1992compact, kaniuth2013induced}.  
In particular, we mention the recent paper \cite{currey2016integrable}, 
in which the Kirillov-Bernat correspondence between the dual and coadjoint orbits is used, 
to obtain sufficient and necessary conditions for integrable wavelet transforms 
phrased in terms of dual orbit spaces $\mathcal{O} \subset \RHat^d$. 
Building on \cite{kaniuth1996minimal}, the results in \cite{currey2016integrable} paved the way 
for the notion of an integrably admissible dilations group as used in this paper. 

Integrable representations also play an essential role in the coorbit space theory  
developed by Feichtinger and Gr\"ochenig 
\cite{feichtinger1989banach1, feichtinger1989banach2, feichtinger1988unified} 
and it is in this context that we investigate the integrable wavelet transform in this paper.
The coorbit spaces associated to semi-direct products are defined in terms of decay properties
of the (extended) matrix coefficients $W_{\psi} f (x,h) = \langle f, \pi (x,h) \psi \rangle$
and provide a family of Banach function spaces induced by the action of the group. 
The main ingredient for defining coorbit spaces is a reproducing formula as in
\eqref{eq:reproducing_intro}. The formula \eqref{eq:reproducing_intro}
is well-known to hold for any discrete series representation \cite{duflo1976regular, grossmann1985transforms}, but also holds for numerous
reducible square-integrable representations, 
including the representations considered here. 
If, in addition to \eqref{eq:reproducing_intro}, 
the  representation $\pi$ is integrable, 
then the reproducing formula can be discretized 
to obtain atomic decompositions for coorbit spaces  
 \cite{grochenig1991describing}. 

\subsection{Coorbit spaces associated to integrably admissible dilation groups}
In this paper we consider coorbit spaces associated 
to integrably admissible dilation groups.
Since this setting contains the class of irreducibly admissible dilation groups as 
a proper subclass, the quasi-regular representation might be reducible.
The scale of coorbit spaces that we obtain extends therefore the wavelet coorbit spaces 
associated to discrete series representations. 
A key property of the original coorbit spaces 
that is guaranteed by the irreducibility is 
its independence of the choice of the admissible analyzing vector used to define the space.  
This property might fail dramatically if the assumption of irreducibility of the representation is dropped. 
Indeed, in $d = 2$, the quasi-regular representation 
$(\pi, L^2 (\RR^2))$ of $G = \RR^2 \rtimes \RR^+$ acts reducibly and 
\cite[Section 2.1]{fuehr2015coorbit} provides an example of two integrably admissible vectors
for which the associated coorbit spaces do not coincide. 
This indicates that for
reducible representations additional considerations
are required in defining the coorbit spaces. 
More precisely, to compensate for the reducible representations, 
the requirements on the space of analyzing vectors need to be more restrictive 
than in the setting of discrete series representations. 
It turns out that this is the only formal adjustment needed to obtain a theory 
with the same features as for the irreducible setup, 
although the arguments used for general integrably admissible groups 
are at places more involved.
A key example that can be covered in the setup considered here, 
but not by the discrete series case, 
is the space induced by the  action of $G = \RR^d \rtimes H$
with $H = \exp(\RR A)$ being the one-parameter subgroup generated by a suitable matrix  
$A \in \mathrm{GL}(d, \mathbb{R})$. 
The obtained space turns out to coincide, up to suitable identifications,
with the well-known anisotropic Besov spaces \cite{bownik2005atomic, barrios2011characterizations}.

The representation-theoretic nature of the definition of coorbit spaces, 
together with a reproducing formula such as \eqref{eq:reproducing_intro},  
allows to transfer questions concerning these spaces 
into related questions concerning functions on the group $G = \RR^d \rtimes H$. 
This has the advantage that questions that do not rely on the geometry underlying $G$
can be approached in a unified manner using general tools from abstract harmonic analysis.
On the other hand, questions that rely on the structure of $G = \RR^d \rtimes H$,
and hence of $H$,
might not even be well-posed in the usual realization of the coorbit space.
Canonical questions of this type are, for example, the dilation-invariance of the 
coorbit space under the action of dilation groups other than $H$, 
and the embeddings between coorbit spaces that are defined using different
dilations groups. 
For the investigation of these questions, the realization of 
the coorbit space as a Besov-type space defined by decomposition methods 
\cite{triebel1977general, triebel1978spaces}, 
a so-called \emph{Besov-type decomposition space} 
\cite{stroeckert1979decomposition, feichtinger1985banach}, is highly beneficial 
as this realization encodes the essential properties of the coorbit space that are determined 
by the dual action of the dilation group $H$. 
The
coorbit spaces associated to the affine group $ \RR \rtimes \RR^+$ 
and the isotropic Besov spaces \cite{peetre1976besov, triebel1977fourier} 
 have been identified from the very beginning
 \cite[Section 7.2]{feichtinger1988unified}, 
and is similar in spirit to the $\varphi$-transform 
characterization of Besov spaces by Frazier and Jawerth 
\cite{Frazier1985Besov, Frazier1990Triebel}. 
The coorbit spaces associated to irreducibly admissible dilation groups 
have been identified with suitable Besov-type decomposition spaces only more recently \cite{fuehr2015wavelet}. 
In this paper we extend this identification to the more general class of 
integrably admissible dilation groups. 
As a consequence, the powerful embedding theory 
for Besov-type decomposition spaces developed in \cite{voigtlaender_embeddings} can be used 
to obtain embedding theorems for these coorbit spaces. 
Other applications of this identification are discussed in Section \ref{sec:examples}. 

\subsection{Related work}
\emph{Wavelet coorbit theory.}
The papers that are closest in spirit to the work presented here
are the predecessors \cite{fuehr2016vanishing,fuehr2015coorbit, fuehr2015wavelet} 
by the first named author. The papers \cite{fuehr2016vanishing,fuehr2015coorbit} 
establish explicit and verifiable criteria on continuous wavelet transforms 
for applying the abstract theory of \cite{feichtinger1989banach1, feichtinger1989banach2}. 
In this paper we extend the results of \cite{fuehr2015coorbit}
 to the class of integrably admissible dilation groups. 
In doing so, some auxillary results in \cite{fuehr2016vanishing,fuehr2015coorbit} will be used by citation, 
i.e., the explicit construction of a suitable control weight
(see Section \ref{sec:norm-estimates} below). 
Moreover, some of our proof methods are based on these in \cite{fuehr2016vanishing,fuehr2015coorbit}, 
but they are non-trivial adaptions in general. 
The difficulties arising in these adaptions are caused by the fact that for
irreducibly admissible dilation groups, the dual actions 
have a unique, singly generated open orbit, whereas for the dilation groups 
considered in this paper, the (possibly non-unique) open orbit might be generated by an 
arbitrary compact set rather than just a singleton; see Section \ref{sec:integrably_admissible} for the details.
On the other hand, we like to emphasize that several of the results in this paper are obtained 
relatively easily, compared to the proofs of analogous results in \cite{fuehr2016vanishing,fuehr2015coorbit}.
In part, this is due to recent progress in decomposition space theory.  
In any case, given the fact that coorbit theory is generally viewed as somewhat technical and cumbersome, 
we regard this gain in effectiveness as an important asset of our approach. 

The paper \cite{fuehr2015wavelet} identifies the wavelet coorbit spaces considered in 
\cite{fuehr2016vanishing,fuehr2015coorbit} as certain Besov-type spaces 
defined by decomposition methods in the sense of 
\cite{stroeckert1979decomposition, triebel1978spaces, feichtinger1985banach, feichtinger1987banach}.
The approach and techniques used in \cite{fuehr2015wavelet} are followed in 
the identification established in Section \ref{sec:identification} 
below. Several preliminary results of \cite{fuehr2015wavelet} that do not rely
on special properties of the dilation group are used here by citation. 
However, the extension of several auxiliary results of \cite{fuehr2015wavelet} 
to integrably admissible dilation groups
are quite subtle and technically more involved compared to the original arguments in \cite{fuehr2015wavelet}. 
These additional technicalities are most apparent in the construction of an induced 
cover and a suitable partition of unity subordinate to this cover; see the results in Section \ref{sec:decomposition}. 
Again, these technicalities are caused by the fact that the open dual orbits of the dilation groups 
considered in this paper might be generated by an arbitrary compact set instead of a singleton.
Lastly, we mention that the actual identification of the  spaces
is much simpler and more transparent than in \cite{fuehr2015wavelet} as the reservoir of the coorbit spaces 
considered in this paper
can be canonically identified with the one used to define decomposition spaces, in contrast to  \cite{fuehr2015wavelet}. 
Moreover, we simplify the proof of the norm equivalence by using a delicate density result 
recently established in \cite{romero_invertibility}. 
\\\\
\emph{Coorbit theory for dual pairs.}
A framework for coorbit spaces associated with possibly reducible, non-integrable representations 
have been developed by Christensen and \'{O}lafsson \cite{christensen2009examples, christensen2011coorbit}. 
The spaces in \cite{christensen2009examples, christensen2011coorbit} 
are defined under suitable continuity and smoothness conditions of the representation and
its associated matrix coefficients. In particular, several desired discretization results
have been obtained for this setting in \cite{christensen2012sampling}. 
The coorbit spaces considered in this paper do formally not fit in the framework for dual pairs
 since the space of analyzing vectors we consider does not form a Fr\'{e}chet space, 
which is a standing assumption in \cite{christensen2009examples, christensen2011coorbit}. 
However, the analyzing vectors chosen in this paper are very beneficial for our purposes
as their anti-dual space can be canonically identified with so-called \emph{Fourier distributions} \cite{triebel1977fourier} that are used in defining Besov-type  spaces.  
As in \cite{christensen2009examples, christensen2011coorbit}, 
the basic properties of the coorbit spaces in this paper are also established 
using continuity and smoothness arguments.

\subsection{Aims and contributions}
The central motivation of this paper is to extend, unify and synthesize results from \cite{fuehr2015coorbit,fuehr2015wavelet,currey2016integrable,fuehr_classification}, 
using the language of coorbit spaces and decomposition spaces as developed in \cite{feichtinger1985banach,stroeckert1979decomposition, feichtinger1989banach1}. 
In particular, we generalize the explicit criteria for a wavelet coorbit theory obtained in \cite{fuehr2015coorbit} 
to the setting of integrably admissible dilation groups. 
The fact that we are able to extend a fairly large portion of the results of the precursor papers in comparatively little space, while developing a somewhat novel version of coorbit space theory in the process, is testament both to the versatility and generality of the arguments in the original sources \cite{grochenig1991describing,feichtinger1989banach2, triebel1978spaces, triebel1977general} underlying our adaptation, and to the recent progress in the theory of Besov-type decomposition spaces, e.g., in \cite{voigtlaender_embeddings,voigtlaender_Sobolev,romero_invertibility}. 
As will be seen below, our simpler approach is made possible by the topological conditions underlying the notion of an integrably admissible 
dilation group, which are weak enough to cover a wide variety of settings, 
but sufficiently stringent to provide a class of analyzing wavelets that is very convenient to work with. 

\subsection{Organization}

The paper is organized as follows. 
In Section \ref{sec:wavelet_integrably} the class of integrably admissible dilation groups is introduced
Moreover, several important properties of their associated wavelet transforms are obtained.
Section \ref{sec:coorbit} is devoted to the coorbit spaces associated to 
integrably admissible dilation groups and their basic properties. 
In particular, it is demonstrated that the standard results on discretization 
of convolution operators can be applied to discretize the reproducing formula 
and obtain atomic decompositions. 
Section \ref{sec:decomposition} consists of the construction 
of induced covers and an explicit partition of unity subordinate 
to such covers.
The identification of a coorbit space as a Besov-type decomposition space
is carried out in Section \ref{sec:identification}. Examples (of classes of)
coorbit spaces associated with integrably admissible dilation groups are provided in Section
 \ref{sec:examples}. In particular, it is shown that the coorbit spaces considered in this paper 
 coincide with the original coorbit spaces provided that the quasi-regular representation is a discrete series representation. 
 
 \subsection{Notation and normalization}
Given functions $f,g : X \to [0,\infty)$, we write $f \lesssim g$ provided there exists a constant
$C>0$ such that $f(x) \leq C g(x)$ for all $x \in X$. We write $f \asymp g$ for $f \lesssim g$ 
and $g \lesssim f$. 
For a matrix $h \in \mathrm{GL}(d, \mathbb{R})$, we denote by $\|h\|_{\infty}$ 
the operator norm of the induced map $h : \mathbb{R}^d \to \mathbb{R}^d$. 
We let $|\cdot | : \mathbb{R}^d \to \mathbb{R}$ be the Euclidean norm on $\mathbb{R}^d$. 

For an open subset $\mathcal{O} \subset \mathbb{R}^d$, 
the  space of compactly supported, smooth functions is denoted by 
$\mathcal{D} (\mathcal{O}):= C_c^{\infty} (\mathcal{O})$. The space of Schwartz functions 
is denoted by $\mathcal{S} (\mathbb{R}^d)$. 
The space of distributions on $\mathcal{O}$ is denoted by $\mathcal{D}'(\mathcal{O})$ 
and the space of tempered distributions by $\mathcal{S}'(\mathbb{R}^d)$. 
The Fourier transform $\mathcal{F} : \mathcal{S}(\mathbb{R}^d) \to \mathcal{S}(\mathbb{R}^d)$ 
is normalized as $\hat{f} (\xi) = \int_{\mathbb{R}^d} f(x) e^{-2\pi ix \cdot \xi} \; dx$.
Its inverse $\mathcal{F}^{-1} f := \hat{f}(- \cdot)$ 
will also be denoted by $\check{f}$. 
Similar notations will be used for the unitary Fourier-Plancherel transform 
$\mathcal{F} : L^2 (\mathbb{R}^d) \to L^2 (\mathbb{R}^d)$ and its inverse.

\section{Wavelet transforms with semi-direct products} 
\label{sec:wavelet_integrably}

For a linear Lie group $H \leq \mathrm{GL}(d, \mathbb{R})$, 
define $G = \RR^d \rtimes H$. The admissibility of a vector $\psi \in L^2 (\RR^d)$
can be conveniently characterized by use of the \emph{dual action} of $H$
on the Fourier domain $\RHat^d \cong \mathbb{R}^d$, defined by 
$H \times \RHat^d \ni (h, \xi) \mapsto h^T \xi$, where $h^T$ denotes the transpose of $h \in H$.

\begin{lemma}[\cite{fuehr2002continuous, laugesen2002characterization}]
A vector $\psi \in L^2 (\mathbb{R}^d)$ is  admissible if, and only if,
\begin{align} \label{eq:Calderon}
 \int_H | \widehat{\psi} (h^T \xi)|^2 \; d\mu_H (h) = 1
\end{align}
for a.e. $\xi \in \RHat^d$. 
\end{lemma}

Given two vectors $\psi_1, \psi_2 \in L^2 (\mathbb{R}^d)$ satisfying
 the admissibility condition \eqref{eq:Calderon},  
the property that the wavelet transform 
$W_{\psi_i} : L^2 (\mathbb{R}^d) \to L^2 (G)$, with $i \in \{1,2\}$,
isometrically intertwines $(\pi, L^2 (\mathbb{R}^d))$ and the left-regular representation 
$(\lambda_{G}, L^2 (G))$, 
yields the reproducing formula
\begin{align} \label{eq:L2_reproducingformula}
 W_{\psi_1} f \ast W_{\psi_2} \psi_1 = W_{\psi_2} f 
\end{align}
for all $f \in L^2 (\RR^d)$. 

\subsection{Integrably admissible dilation groups} \label{sec:integrably_admissible}
The following definition introduces the class of dilation groups 
that will be treated in this paper. 

\begin{definition} \label{def:integrably_admissible}
A closed subgroup $H \leq \mathrm{GL}(d, \mathbb{R})$ is called \emph{integrably admissible with essential frequency support $\mathcal{O}$}
if  $\mathcal{O} \subset \RHat^d$ is an open set satisfying:
\begin{enumerate}[(i)] 
\item The set $\mathcal{O}$ is co-null and $H^T$-invariant;
\item The dual action of $H$ on $\mathcal{O}$ is proper, i.e., 
for all compact subsets $K \subset \mathcal{O}$, the set
\[
\{ (h, \xi) \in H \times \mathcal{O} \; : \; (h^T \xi, \xi) \in K \times K \} 
\subset H \times \mathcal{O} 
\]
is compact;
\item \label{it:compactness_criterium} 
There exists a compact subset $C \subset \mathcal{O}$ such that
$\mathcal{O} = H^T C$.
\end{enumerate}
\end{definition}

\begin{remark}
\begin{enumerate}[(a)]
\item
 The condition that the dilation group $H \leq \mathrm{GL}(d, \mathbb{R})$ 
 is closed in the relative topology is rather natural 
 and is necessary for the existence of an admissible vector, 
 cf. \cite[Proposition 5.1]{fuehr2005abstract} and \cite[Proposition 5]{MR1633179}.
\item
The compactness criterion \eqref{it:compactness_criterium} 
in Definition \ref{def:integrably_admissible} 
is equivalent to the orbit space $\mathcal{O} / H^T$ being compact with respect to the quotient topology, see
\cite[Corollary 4.2]{kaniuth1996minimal}. 
Condition (ii) therefore entails, in particular,
that the isotropy subgroups $H_{\xi} := \{ h \in H \; | \; h^T \xi = \xi \}$ 
are compact for all $\xi \in \mathcal{O}$. 
Moreover, the orbit maps $p_{\xi} : H \to \mathcal{O}, \; h \mapsto h^T \xi$
are proper for  $\xi \in \mathcal{O}$.  
\end{enumerate}
\end{remark}

It is currently not well-understood whether the essential frequency support associated to an integrably admissible matrix group is uniquely defined by the group or not. It is conceivable that given such a group, two different sets $\mathcal{O}$ and $\mathcal{O}'$ could serve as essential frequency support, which would necessarily only differ by a nullset. However, note that essential frequency support is described in terms of topological conditions, which cannot be expected to be stable under addition or removal of a set of measure zero. This somewhat subtle point is relevant, since the essential frequency support enters into the decomposition space description of coorbit spaces. 

In order to state a convenient characterization of integrably admissible dilation groups, 
we denote, following \cite{palais1961on}, for given $Y, Z \subset \mathbb{R}^d$, 
\[
((Y,Z)) := \{ h \in H \; | \; h^T Y \cap Z \neq \emptyset \}. 
\]
The following result is \cite[Proposition 12]{currey2016integrable}, rephrased in 
the terminology of the current paper. 

\begin{lemma}[\cite{currey2016integrable}] \label{lem:characterization_integrably}
Let $H \leq \mathrm{GL}(d,\mathbb{R})$ and let $\mathcal{O} \subset \RHat^d$ be open, co-null and $H^T$-invariant. 
Then the following are equivalent:
\begin{enumerate}
\item[(i)] The group $H$ is integrably admissible with essential frequency support $\mathcal{O}$.
\item[(ii)] There exists some open, relatively compact $C \subset \mathcal{O}$ with $H^T C = \mathcal{O}$ and for any 
such set $C$, the set $((C,C))$ is relatively compact in $H$. 
\end{enumerate}
\end{lemma}

\begin{corollary} \label{cor:KKcp}
If $H \leq \mathrm{GL}(d,\mathbb{R})$ is integrably admissible with essential frequency support $\mathcal{O}$, then, given compact sets $C_1, C_2 \subseteq \mathcal{O}$, the set
$((C_1, C_2))$ is relatively compact in $H$. 
\end{corollary}
\begin{proof}
Let $C \subset \mathcal{O}$ be a compact set
such that $H^T C = \mathcal{O}$. Then  $((C, C))$ is relatively compact by Lemma \ref{lem:characterization_integrably}. For $i \in \{1,2\}$, the compactness of $C_i$
yields finite sets
$F_i \subseteq H$ such that $C_i \subset F_i C$. 
But $((C_1, C_2)) \subseteq F_2 ((C,C)) F_1^{-1}$, and thus $((C_1, C_2))$ is relatively compact. 
\end{proof}

Unless otherwise specified, throughout this paper, the subgroup $H$ of  $\mathrm{GL}(d, \mathbb{R})$ 
will denote an integrably admissible dilation group with essential frequency support $\mathcal{O}$. 

We next list several classes of integrably admissible dilation groups. The associated coorbit spaces are discussed in more detail in Section \ref{sec:examples}. 

\begin{example} \label{ex:integrablyadmissible}
\begin{enumerate}[(a)]
\item If $H \leq \mathrm{GL}(d, \mathbb{R})$ is (irreducibly) \emph{admissible}, i.e., 
  there exists a single open orbit 
$\mathcal{O} = H^T \xi_0$ of full Lebesgue measure for which the associated 
 isotropy group $H_{\xi_0}$ is compact in $H$,
 then $H$ is readily seen to be integrably admissible.  
The irreducibly admissible dilation groups are precisely the ones for which
the associated quasi-regular representation $(\pi , L^2 (\RR^d))$ of $G$ 
is a discrete series representation \cite{fuehr2010generalized}. 
A full classification, up to conjugacy classes, in dimension three can be found in 
\cite{currey_classification}. 
Explicit necessary and sufficient conditions for abelian dilation groups 
$H \leq \mathrm{GL}(d, \mathbb{R})$ can be found in \cite{bruna2015characterizing}. 
\item If $H = \exp(\RR A)$, then $H$ is integrably admissible if, and only if, 
the real parts of all eigenvalues of $A$ are either strictly negative or strictly positive, 
see \cite{kaniuth2013induced, grochenig1992compact}. 
The essential frequency support can then be taken as $\mathcal{O} = \RR^d \setminus \{0\}$. 
\item \label{ex:two-parameter} If $d = 3$ and $H \leq \mathrm{GL}(d, \mathbb{R})$ is connected and abelian, 
then a complete classification of the integrably admissible dilation groups is given in \cite[Proposition 18]{currey2016integrable}.  
As a concrete example, we mention $H = \exp(\mathbb{R} A) \exp(\mathbb{R}B)$, with 
$$
A = \left[ \begin{array}{ccc} 1 && \\ &0&\\&&\alpha\end{array} \right]
$$
and 
$$
B = \left[ \begin{array}{ccc} 0 && \\ &1&\\&&\beta \end{array} \right].
$$ 
Then $H$ is an abelian, two-dimensional closed subgroup of $GL(3,\mathbb{R})$. By \cite[Proposition 4.1]{currey2016integrable}, $H$ is integrably admissible if and only if $\alpha \beta = 0$ and $\alpha + \beta >0$, i.e., one of the two parameters is zero, and the other one strictly positive. The essential frequency support is given by 
\[
 \mathcal{O} = \{ (x_1,x_2,x_3) : x_1 \not= 0 \not=  x_2^2+x_3^2 \} \mbox{ for } \alpha = 0
\] and
 \[
 \mathcal{O} = \{ (x_1,x_2,x_3) : x_1^2+x_3^2 \not= 0 \not=  x_2 \} \mbox{ for } \beta = 0.
\] 

Picking parameters $\alpha, \beta$ with $\alpha \beta < 0$ provides an intriguing example of semidirect product groups $\mathbb{R}^3 \rtimes H$ with the property that the support of the quasi-regular representation $\pi$ is an open compact subset of the dual. However, it is currently unknown whether there exists an admissible vector $\psi$ such that $\Delta_G^{-1/2} \mathcal{W}_\psi \psi \in L^1(G)$. Since in this case $H$ is not integrably admissible, the methods for constructing these vectors that this paper relies on are not available.  
\end{enumerate}
\end{example}

The significance of an integrably admissible dilation group is that it guarantees 
the existence of suitable admissible vectors, as proven in \cite[Proposition 10]{currey2016integrable}:

\begin{theorem}[\cite{currey2016integrable}] \label{thm:admissible_existence}
Suppose $\mathcal{O} = H^T C$ is an essential frequency support of $H \leq \mathrm{GL}(d, \mathbb{R})$.  
Then there exists an admissible vector 
$\psi \in \mathcal{F}^{-1} (C_c^{\infty} (\mathcal{O}))$, with $\widehat{\psi}^{-1} (\mathbb{C} \setminus \{0\}) \supset C$. 
\end{theorem}

\subsection{Extended matrix coefficients}
This section is concerned with extending the action of the quasi-regular representation 
to suitable distribution spaces. In particular, a point-wise estimate for the decay of the extended matrix 
coefficients will be obtained and the reproducing formula \eqref{eq:L2_reproducingformula}
will be extended. 

\begin{definition}
The space of \emph{analyzing vectors} $\mathcal{S}_{\mathcal{O}} (\mathbb{R}^d)$
is defined as the inverse Fourier image 
$\mathcal{S}_{\mathcal{O}} (\mathbb{R}^d) := \mathcal{F}^{-1} (\mathcal{D}(\mathcal{O}))$,
where $\mathcal{D}(\mathcal{O}) := C_c^{\infty} (\mathcal{O})$ is equipped with the usual inductive limit topology.
The space $\mathcal{S}_{\mathcal{O}} (\mathbb{R}^d)$ will be equipped with 
the unique topology making $\mathcal{F}^{-1} : \mathcal{D} (\mathcal{O}) \to \mathcal{S}_{\mathcal{O}} (\mathbb{R}^d)$ into a homeomorphism. 
The anti-dual space of $\mathcal{S}_{\mathcal{O}} (\mathbb{R}^d)$, 
i.e., the space of all continuous conjugate-linear functionals, 
 will be written as $\SOdual$.
The space $\SOdual$ is assumed to be equipped with the weak$^*$-topology
and the sesquilinear pairing
$\langle \cdot , \cdot \rangle : \SOdual \times \mathcal{S}_{\mathcal{O}} (\mathbb{R}^d) \to \mathbb{C}$.  
\end{definition}

The (extended) Fourier transform $\mathcal{F} f := \widehat{f} \in \mathcal{D}' (\mathcal{O})$ of $f \in \SOdual$ is defined 
by duality; i.e.,
\[
\mathcal{F} : \SOdual \to \mathcal{D}'(\mathcal{O}), \quad f \mapsto \hat{f} := f \circ \mathcal{F}^{-1}.
\]

We mention the following simple properties of $\mathcal{S}_{\mathcal{O}} (\RR^d)$ without proof. 

\begin{lemma}
\begin{enumerate}[(i)]
\item The space $\mathcal{S}_{\mathcal{O}} (\mathbb{R}^d)$ is $\pi$-invariant. 
\item The space $\mathcal{S}_{\mathcal{O}} (\mathbb{R}^d)$ is weak$^*$-dense in 
$\SOdual$
\end{enumerate}
\end{lemma}

The action of $\pi$ can be extended from $\mathcal{S}_{\mathcal{O}} (\mathbb{R}^d)$ to $\SOdual$ by the contragradient: For $f \in \SOdual$ 
and $\psi \in \mathcal{S}_{\mathcal{O}} (\mathbb{R}^d)$, 
\[
\langle \pi(g) f, \psi \rangle = \langle f, \pi(g^{-1}) \psi \rangle, \quad g \in G. 
\]
The (extended) wavelet transform $W_{\psi} f$ of $f \in \SOdual$, relative to $\psi \in \mathcal{S}_{\mathcal{O}} (\mathbb{R}^d)$, is defined by
\[
W_{\psi} f : G \to \mathbb{C}, \quad g \mapsto \langle f , \pi(g) \psi \rangle. 
\]
Point-wise estimates of the coefficient decay are provided by the following result.
Here, given a compact set $K \subset \RR^d$ and  $N >0$, 
the semi-norm of $f \in C^{\infty}(K)$ is as usual denoted by 
$\| f \|_{K,N} := \max_{|\alpha| \leq N} \| (\partial^{\alpha} f) \cdot \mathds{1}_K \|_{\infty}$. 

\begin{lemma} \label{lem:coefficient_estimate}
Let $\psi \in \mathcal{S}_{\mathcal{O}} (\mathbb{R}^d)$ with $K_1 := \supp \widehat{\psi}$.
\begin{enumerate}[(i)]
\item Suppose $\phi \in \mathcal{S}_{\mathcal{O}} (\mathbb{R}^d)$ with $K_2 := \supp \widehat{\phi}$.
Then, for all $N \in \mathbb{N}$, there exists a $C = C(N) > 0$ such that
\[
| W_{\psi} \phi (x,h)| \leq C (1 + |x|)^{-N} (1 + \|h \|_{\infty})^{N+1/2} \| \widehat{\psi} \|_{K_1, N} \| \widehat{\phi} \|_{K_2, N} 
\mathds{1}_{((K_1, K_2))} (h) 
\]
for all $(x,h) \in \mathbb{R}^d \rtimes H$. 
\item Suppose $f \in \SOdual$. Then there exist measurable functions
$m_1: H \to \mathbb{R}^+$ and $m_2 : H \to \mathbb{N}_0$ that are bounded on compact sets,
such that
\[
|W_{\psi} f (x,h) | \leq m_1(h) (1 + |x|)^{m_2(h)} (1 + \| h \|_{\infty})^{m_2(h)} \| \widehat{\psi} \|_{K_1, m_2(h)}
\]
for all $(x,h) \in \mathbb{R}^d \rtimes H$
\end{enumerate}
\end{lemma}
\begin{proof}
The proof of (i) follows by similar arguments as \cite[Theorem 3.7]{fuehr2015coorbit}, 
and hence we only provide a sketch. 
We first write 
\begin{align} \label{eq:wavelet_inversefourier}
\mathcal{W}_{\psi} \phi (x,h) = \left( |{\rm det}(h)|^{1/2} \widehat{\phi}
\cdot (D_h \overline{\widehat{\psi}}) \right)^{\vee}(x),
\end{align}
where $D_h \widehat{\psi} := \widehat{\psi} (h^T \cdot)$.
Thus, if $K_1 \cap h^{-T} K_2 = \emptyset$, 
then the right-hand side of \eqref{eq:wavelet_inversefourier} is zero.
That is, if $h \notin ((K_1, K_2))$, then  
$\mathcal{W}_{\psi} \phi (\cdot,h) \equiv 0.$

In order to get the point-wise decay estimate with respect to $x \in \mathbb{R}^d$, 
we first estimate the partial derivatives of \eqref{eq:wavelet_inversefourier}. 
Applications of Leibniz' formula and the chain rule yields 
\begin{eqnarray*}
\left|\partial^\alpha (\widehat{\phi} \cdot D_h \widehat{\psi}) (\xi)
\right| & \le &  \sum_{\gamma+\beta =\alpha}  \frac{\alpha!}{\beta! \gamma!}  \left|\left(
\partial^\beta \widehat{\phi} \right)(\xi) \right|~ \left| \left(
\partial^\gamma (D_h \widehat{\psi}) \right)(\xi) \right|  
\\
& \le & C(d,\alpha) (1+\|h \|_\infty)^{|\alpha|} \sum_{|\beta|, |\gamma| \le |\alpha|}  
 \left|\left(
\partial^\beta \widehat{\phi} \right)(\xi) \right| \cdot \left| \left( \partial^\gamma \widehat{\psi} \right) (h^T \xi) \right|
\\ 
& \le & C(d,\alpha)  (1+\|h \|_\infty)^{|\alpha|} \| \widehat{\psi} \|_{K_1,|\alpha|} \| \widehat{\phi } \|_{K_2,|\alpha|} \mathds{1}_{K_2}(\xi).
\end{eqnarray*} 
By integrating this estimate over $\xi \in \RHat^d$, we obtain 
\begin{eqnarray*}
\left| \mathcal{W}_{\psi} \phi (x,h) \right| &  \le  & C(d) (1+|x|)^{-N}  \max_{|\alpha| \le N} \left\| 
\partial^\alpha  \left( \left( |{\rm det}(h)|^{1/2} \widehat{\phi}
\cdot (D_h \overline{\widehat{\psi}}) \right) \right) \right\|_{L^1} \\
& \le &   C(d,N,K_2)  (1+|x|)^{-N} (1+\|h \|_\infty)^{N+1/2} \| \widehat{\psi} \|_{K_1,N}  \| \widehat{\phi } \|_{K_2,N} \mathds{1}_{((K_1, K_2))} (h),
\end{eqnarray*}  where we used Hadamard's inequality to estimate the determinant against the norm.

(ii) Let $f \in \SOdual$. 
Then $\widehat{f} \in \mathcal{D}'(\mathcal{O})$
and hence, for every
compact $K \subset \mathcal{O}$, there exists $ N(K) \in \mathbb{N}$ and a constant $ C(K) > 0$
such that
\begin{equation} \label{eqn:estimate_functional_f}
|\langle f , \varphi \rangle | = |\langle \widehat{f}, \widehat{\varphi} \rangle| \leq C(K) \| \widehat{\varphi} \|_{K, N(K)}
\end{equation}
for every $\varphi \in \mathcal{S}_{\mathcal{O}} (\mathbb{R}^d)$ 
with $\supp \widehat{\varphi} \subset K$, e.g., see \cite[Theorem 6.8]{rudin1973functional}. 
We want to apply this to $\varphi = \pi (x,h) \psi$. For this purpose, we first estimate, via another application of the Leibniz formula, for all $\xi \in K$,
\begin{eqnarray*}
 \left|  \partial^\alpha \mathcal{F} \left( \pi(x,{\rm Id}) \psi \right)  (\xi) \right| & = & \left|  \partial^\alpha \left( \exp(2 \pi i \langle x, \cdot \rangle ) \widehat{\psi} \right) (\xi) \right| \\
 & \le & C(d,\alpha,K) (1+|x|)^{|\alpha|} \| \widehat{\psi} \|_{K,|\alpha|} ~.
\end{eqnarray*}
By the chain rule argument already employed for part (i), we further get 
\[
 \left| \partial^\alpha \mathcal{F} \left( \pi(0,h) \psi \right)  (\xi) \right| \le  C(d, |\alpha|) (1+\|h\|_\infty)^{|\alpha|+1/2} \| \widehat{\psi} \|_{K, |\alpha|}
\]
Plugging $\varphi = \pi(x,{\rm Id}) \circ \pi(0,h) \psi$ into (\ref{eqn:estimate_functional_f}) then yields 
\[
| W_{\psi} f (x,h) | \leq C (d,K) (1 + |x|)^{N(K)} (1+\|h\|_\infty)^{N(K)+1/2} \| \widehat{\psi} \|_{K_1, N(K)}
\]
provided that $K \supset \supp( \widehat{\pi(x,h) \psi}) = h^{-T} K_1$. 

Lastly, to define the functions $m_1 : H \to \mathbb{R}^+$ and 
$m_2 : H \to \mathbb{N}_0$, take an increasing covering $(A_{\ell})_{\ell \in \mathbb{N}}$
of $\mathcal{O}$ consisting of open, relatively compact sets $A_{\ell} \subset \RHat^d$.
Let $\ell (h) := \inf \{n \in \mathbb{N} \; | \; h^{-T} K_1 \subset \overline{A_n} \}$ for $h \in H$.
Then every compact set $K \subset \mathcal{O}$ is 
contained in some $A_\ell$, hence  $\ell$ is a well-defined measurable function that is bounded on compact sets. The same then holds for $m_1,m_2$, defined by $m_1 (h) = C(\overline{A_{\ell(h)}})$ 
and $m_2(h) = N(\overline{A_{\ell(h)}})+1/2$. These functions yield the desired estimate. 
\end{proof}

As a first application of Lemma \ref{lem:coefficient_estimate}, 
 the $L^2$-reproducing formula 
\eqref{eq:L2_reproducingformula} will be extended 
to the whole distribution space $\SOdual$. 

\begin{lemma} \label{lem:technical_lemma1}
Let $\psi_1 \in \mathcal{S}_{\mathcal{O}} (\mathbb{R}^d)$ be admissible. 
Then the following assertions hold:
\begin{enumerate}[(i)] 
\item For any $ \psi_2 \in \mathcal{S}_{\mathcal{O}} (\RR^d)$, the map 
\[
\SOdual \ni f 
\mapsto 
\int_G \langle f, \pi(x,h) \psi_1 \rangle \langle \pi(x,h) \psi_1, \psi_2 \rangle \; d_G (x,h)
\in \mathbb{C}
\]
is weakly continuous. 
\item For any two admissible $\psi_1, \psi_2 \in \mathcal{S}_{\mathcal{O}} (\RR^d)$ and any $f \in \SOdual$, 
the reproducing formula $W_{\psi_2} f = W_{\psi_1} f \ast W_{\psi_2} \psi_1$
holds.
\end{enumerate}
\end{lemma} 
\begin{proof}
(i) It suffices to prove that
\[
\int_G \langle f, \pi(x,h) \psi_1 \rangle \langle \pi(x,h) \psi_1, \psi_2 \rangle \; d_G (x,h) = \langle f, \psi_2 \rangle.
\]
For this, consider the Fourier-side reproducing formula
\begin{align} \label{eq:Fourier_reproducingformula}
\widehat{\psi_2}  = \int_G W_{\psi_1} \psi_2 (x,h) \widehat{[\pi(x,h) \psi_1]} \; d_G (x,h). 
\end{align}
The integral \eqref{eq:Fourier_reproducingformula} converges absolutely by Lemma \ref{lem:coefficient_estimate}(i). 
Since the action $H \times \mathbb{R}^d \ni (h, \xi) \mapsto h^T \xi$ is proper and
$\widehat{\psi_1} \in \mathcal{D} (\mathcal{O})$, 
there exists a compact set 
$K \subset \mathcal{O}$ such that
\[
\supp \bigg( W_{\psi_1} \psi_2 (x,h) \widehat{[\pi(x,h) \psi_1]} \bigg) \subset K
\]
for all $(x,h) \in G$. Given $N \in \mathbb{N}$ sufficiently large, the decay estimate in
Lemma \ref{lem:coefficient_estimate}(i) guarantees that the map
\[
(x,h) \mapsto  W_{\psi_1} \psi_2 (x,h) [\widehat{ \pi (x,h) \psi_1}]
\]
from $G$ into the Banach space $(C^N (K), \| \cdot \|_{K,N} )$ is Bochner integrable,
yielding that 
\[
\int_G W_{\psi_1} \psi_2 (x,h) [\widehat{ \pi (x,h) \psi_1 }] \; d\mu_G (x,h) \in C^N (K)
\]
is well-defined. 
An application of \cite[Theorem 6.6]{rudin1973functional}, 
together with the Hahn-Banach theorem, 
entails that $\widehat{f} : \mathcal{D}(\mathcal{O}) \cap C^N (K) \to \mathbb{C}$ 
induces a continuous extension to all of $C^N (K)$ for 
sufficiently large $N \in \mathbb{N}$.  
This justifies to interchange Bochner integration and 
the evaluation of continuous
functionals, e.g., see \cite[V.5, Corollary 2]{yosida1980functional}. 
Therefore,
\begin{align*}
\langle f, \psi_2 \rangle &= \langle \widehat{f}, \widehat{\psi_2} \rangle 
= \int_G \big\langle \widehat{f} , \mathcal{F} (W_{\psi_1} \psi_2 (x,h) \cdot \pi (x,h) \psi_1 ) \big\rangle \; d_G (x,h)  \\
&= \int_G \big\langle f, \pi(x,h) \psi_1 \rangle \langle \pi(x,h) \psi_1, \psi_2 \big\rangle \; d_G (x,h), 
\numberthis \label{eq:weak_continuity}
\end{align*}
which shows (i). 

(ii) Let $f \in \SOdual$. 
By density, 
there exists a net $(\phi_{\alpha})_{\alpha \in \Lambda}$ 
of functions $\phi_{\alpha} \in \mathcal{S}_{\mathcal{O}} (\mathbb{R}^d)$ 
converging weakly to $f$. 
Equation \eqref{eq:L2_reproducingformula} yields that
 $W_{\psi_1} \phi_{\alpha} \ast W_{\psi_2} \psi_1 =  W_{\psi_2} \phi_{\alpha}$ 
for all $\alpha \in \Lambda$. 
Since the mapping
\[ f \mapsto \int_G \langle \pi(y^{-1}) f, \pi (g) \psi_1 \rangle 
\langle \pi(g) \psi_1, \psi_2 \rangle \; d\mu_G (g) = 
(W_{\psi_1} f \ast W_{\psi_2} \psi_1) (y) \] 
is weakly continuous by part (i), 
it follows that
\[ \lim_{\alpha \in \Lambda} \big(W_{\psi_1} \phi_{\alpha} \ast W_{\psi_2} \psi_1\big) (y) = (W_{\psi_1} f \ast W_{\psi_2} \psi_1) (y)
\]
 for every $y \in G$. 
On the other hand, clearly
$\lim_{\alpha \in \Lambda} W_{\psi_2} \phi_{\alpha} = W_{\psi_2} f$.
Combining both identities gives therefore
$W_{\psi_2} f = W_{\psi_1} f \ast W_{\psi_2} \psi_1$, 
as required. 
\end{proof}

\subsection{Norm estimates} \label{sec:norm-estimates}

The left Haar measure $\mu_G$ on $G = \mathbb{R}^d \rtimes H$ 
is given by $\mu_G (x,h) = |\det h|^{-1} \cdot dx dh$,
 and the modular function $\Delta_G$ on $G$ is $\Delta_G (x,h) = |\det h |^{-1} \Delta_H (h)$. 
Given a weight $v : G \to \RR^+$, the associated weighted mixed Lebesgue space 
is defined as
\[
L^{p,q}_v (G) := \bigg\{ F \in L^1_{\loc} (G) \; : \; \| f \|_{L^{p,q}_v} < \infty \bigg\}
\]
with 
\[
\|f \|_{L^{p,q}_v} := \bigg( \int_H \bigg( \big\| v(\cdot, h) f (\cdot, h) \big\|_{L^p (\RR^d)} \bigg)^{q} 
\frac{dh}{|\det h|} \bigg)^{1/q} < \infty 
\]
for $p \in [1,\infty]$ and $q \in [1,\infty)$ and
$
\| f \|_{L^{p,\infty}_v} := \esssup_{h \in H} \big( \big\| w(\cdot, h) f(\cdot, h) \big\|_{L^p (\RR^d)} \big). 
$
As usual, the conjugate $p'$ of $p \in (1,\infty)$ is defined as $p' := \frac{p}{p-1}$, and
$1' := \infty$ and $\infty' := 1$.

Throughout this paper, any weighting function on $G$ will be assumed to 
be \emph{admissible} in the following sense. 

\begin{definition}
A continuous weighting function $v : G \to (0, \infty)$ is called \emph{admissible} 
if
\begin{enumerate}[(i)]
\item The submultiplicativity condition is satisfied, i.e.,
$v((x_1, h_1) (x_2, h_2)) \leq C_v \cdot v(x_1, h_1) v(x_2, h_2)$ for all
$(x_1, h_1), (x_2, h_2) \in \mathbb{R}^d \rtimes H$ and some $C_v > 0$. 
\item There exists a locally bounded weight $v_0 : H \to (0,\infty)$ 
and an $s \in [0,\infty)$
such that $v(x,h) \leq (1+|x|)^s v_0 (h)$ for all $(x,h) \in \mathbb{R}^d \rtimes H$. 
\end{enumerate}
\end{definition}

Norm estimates for convolution operators play an essential role in the sequel. 
In particular, it will be used repeatedly that $L^{p,q}_v (G)$ forms a Banach convolution module over $L^1_w (G)$ 
for a so-called \emph{control weight} $w : G \to \mathbb{R}^+$. 

\begin{definition}
Let $Y = L^{p,q}_v (G)$ for $p,q \in [1,\infty]$ and an admissible weight $v : G \to \mathbb{R}^+$. 
A weight $w : G \to \mathbb{R}^+$ is called a \emph{control weight} for $Y$ if it satisfies
\[
w (x,h) = \Delta_G (x,h)^{-1} w ((x,h)^{-1})
\]
and
\[
\max \big( \|L_{(x,h)} \|_{Y \mapsto Y}, \|L_{(x,h)^{-1}} \|_{Y \mapsto Y}, \|R_{(x,h)} \|_{Y \mapsto Y},
\|R_{(x,h)^{-1}} \|_{Y \mapsto Y} \Delta_G (x,h)^{-1} \big) \leq w (x,h),
\]
where $L_{(x,h)}$ and $R_{(x,h)}$ denote the left and right translation by $(x,h) \in G$, respectively. 
\end{definition}

Given an admissible weight, 
there exists an associated (admissible) control weight of the same type. The corresponding 
control weight is explicitly given by the following result, 
which is \cite[Lemma 2.3]{fuehr2015coorbit}. 

\begin{lemma}[\cite{fuehr2015coorbit}] \label{lem:control_weight}
Let $p, q \in [1,\infty]$ and let $v : G \to \mathbb{R}^+$ 
be an admissible weight satisfying $v (x,h) \leq (1+|x|)^s v_0(h)$ for $s \in [0,\infty)$. Then there
exists a control weight $w : G \to \mathbb{R}^+$ for $Y = L^{p,q}_v (G)$ satisfying the estimate
\[
w (x,h) \leq (1 + |x|)^s w_0 (h), 
\]
where $w_0 : H \to \mathbb{R}^+$ is defined by
\begin{align*}
w_0 (h) &= \big(v_0(h) + v_0(h^{-1}) \big) \cdot \max \big ( \Delta_G (0,h)^{-\frac{1}{q}}, \Delta_G (0,h)^{\frac{1}{q} - 1} \big) \\
& \quad \quad 
\cdot \big( | \det h |^{\frac{1}{q} - \frac{1}{p}} + |\det h|^{\frac{1}{p} - \frac{1}{q}} \big) \big(1 + \|h\|_{\infty} + \|h^{-1} \|_{\infty} \big)^s
\end{align*}
with the convention $1/\infty := 0$. 
\end{lemma}

\begin{remark}
Let $v : G \to \mathbb{R}^+$ be an admissible weight. 
The control weight $w : G \to \mathbb{R}^+$ for
$Y = L^{p,q}_v (G)$ provided by Lemma \ref{lem:control_weight} depends on $p,q \in [1,\infty]$ and $s \in [0,\infty)$. 
By fixing $s' > 0$ and a submultiplicative, locally bounded weight $v_0 : H \to \mathbb{R}^+$, 
a control weight $w' : G \to \mathbb{R}^+$ can be chosen that works simultaneously for all $p, q \in [1,\infty]$ 
and all $s \in [0, s']$, e.g., the weight 
\begin{align*}
w' (x,h) &= (1+|x|)^{s'} (v_0 (h) + v_0(h)^{-1}) \max \big(1, \Delta_G (0,h)^{-1} \big) \\
& \quad \quad \cdot \big( |\det (h)| + |\det (h)|^{-1} \big) \big(1 + \|h\|_{\infty} + \|h^{-1}\|_{\infty} \big)^{s'} 
\end{align*}
satisfies $w(x,h) \leq w'(x,h)$ for all $(x,h) \in G$, where $w : G \to \mathbb{R}^+$ is any control weight 
obtained via Lemma \ref{lem:control_weight}. 
\end{remark}

For reference purposes, we mention the following well-known properties. 

\begin{proposition} \label{prop:admissible_weight_properties}
Let $v : G \to \mathbb{R}^+$ be an admissible weight and let $p,q \in [1,\infty]$ be arbitrary. 
\begin{enumerate}[(i)]
\item The space $Y = L^{p,q}_v (G)$ is a solid Banach function space. Moreover, left and right
translation act strongly continuously on $L^{p,q}_v (G)$. 
\item There exists an admissible control weight $w : G \to \mathbb{R}^+$ such that 
$L^{p,q}_v (G) \ast L^1_{w} (G) \hookrightarrow L^{p,q}_v (G)$,
with $\| F_1 \ast F_2 \|_{L^{p,q}_v}  \leq \| F_1 \|_{L^{p,q}_v} \|F_2 \|_{L^1_w}$. 
\item Any $F_1 \in L_v^{p,q} (G)$ and $F_2 \in L_{1/v}^{p',q'} (G)$ satisfy the \emph{generalized H\"older inequality}
\[
\bigg| \int_G F_1(g) F_2 (g) \; d\mu_G (g) \bigg| \leq \|F_1\|_{L^{p,q}_v} \|F_2 \|_{L^{p',q'}_{1/v}}, 
\]
where $p',q' \in [1,\infty]$ denote the conjugate exponents of $p,q \in [1,\infty]$. 
\end{enumerate}
\end{proposition}

Lastly, it will be 
shown that the wavelet transform of an analyzing vector belongs to a certain 
amalgam space \cite{holland1975harmonic, fournier1985amalgams, feichtinger1983wiener}. 
In particular, that the quasi-regular representation 
$(\pi, L^2 (\mathbb{R}^d))$ is $v$-integrable for any admissible weight $v$. 

\begin{definition} \label{def:amalgam}
Let $Y = L^{p,q}_v (G)$ for $p,q \in [1,\infty]$ and an admissible weight $v : G \to \RR^+$. 
Let $U \subseteq G$ be a relatively compact unit neighborhood. 
For $F \in L^{\infty}_{\loc}  (G)$, define 
\[ F^{\sharp}_U : G \to [0,\infty], \quad x \mapsto \| \mathds{1}_{Ux} F \|_{L^{\infty}}.
\]
The \emph{(right-sided) Wiener amalgam space} is defined by
\[
W^R (L^{\infty}, Y)(G) := \{ F \in L^{\infty}_{\loc} (G) \; | \; F^{\sharp}_U \in Y \}
\]
and equipped with the norm $\|F\|_{W^R(L^{\infty}, Y)} = \|F_U^{\sharp} \|_{Y}$. 
\end{definition}

\begin{remark}
The amalgam spaces  in Definition \ref{def:amalgam} 
form Banach spaces and are independent of the chosen neighborhood,
with equivalent norms for different choices, e.g., see \cite{feichtinger1983wiener}. 
\end{remark}

\begin{proposition} \label{prop:amalgam_estimate}
Suppose $\psi \in \mathcal{S}_{\mathcal{O}} (\mathbb{R}^d)$. 
Then $W_{\psi} \psi \in W^R (L^{\infty}, L^{p,q}_v) (G)$ for all $p, q \in [1,\infty]$
and any admissible weighting function $v : G \to \RR^+$. 
In particular, $W_{\psi} \psi \in L^1_v (G)$. 
\end{proposition} 
\begin{proof}
Let $K := \supp \widehat{\psi}$. 
For some open, relatively compact unit neighborhood $V \subset H$, 
define the unit neighborhood $U := B_1 (0) \times V$ of $G$. 
Let $(x,h) \in \mathbb{R}^d \rtimes H$ be fixed. Then, for arbitrary $N \in \mathbb{N}$, 
an application of Lemma \ref{lem:coefficient_estimate}(i)
gives a constant $C = C(N) > 0$ such that
\begin{align*}
(W_{\psi} \psi)_U^{\sharp} (x,h) 
& \leq \sup_{(y, k) \in U} | W_{\psi} \psi (y + k x, kh)| \\
& \leq C \| \widehat{\psi} \|^2_{K, N} \sup_{(y,k) \in U} (1 + |y + k x|)^{-N} (1 + \| k h \|_{\infty})^{N+1/2} \mathds{1}_{((K, K))} (kh). \numberthis \label{eq:amalgam_pointwise}
\end{align*}
To estimate \eqref{eq:amalgam_pointwise}, 
let $(y, k) \in U = B_1 (0) \times V$ be fixed.
First, combining the estimates $|x| \leq \| k^{-1} \|_{\infty} |kx| \leq \sup_{k \in \overline{V}} \| k^{-1} \|_{\infty} |kx|$ 
and $|y + kx| \geq |kx|-|y| \geq |x| \big( \sup_{k \in \overline{V}} \| k^{-1} \|_{\infty} \big)^{-1} - 1$ gives
\begin{align} \label{eq:estimate_expanded}
1+|x| \leq 1 + \sup_{k \in \overline{V}} \| k^{-1} \|_{\infty} (1 + |y + kx|) \leq (1 + \sup_{k \in \overline{V}} \| k^{-1} \|_{\infty}) (1 + |y + kx|). 
\end{align}
On the other hand, we have
$ (1 + \| k h \|_{\infty}) \leq (1 + \sup_{k \in \overline{V}} \|k \|_{\infty}) (1 + \| h \|_{\infty}).
$
Therefore, defining $C' :=  C (1 + \sup_{k \in \overline{V}} \| k^{-1} \|)^{-N}  (1 + \sup_{k \in \overline{V}} \| k \| )^\frac{1}{N+1/2}$, we can estimate \eqref{eq:amalgam_pointwise} by
\begin{align} \label{eq:local_estimate}
(W_{\psi} \psi)_U^{\sharp} (x,h)  \leq C' 
 (1 + |x|)^{-N} (1 + \| h \|_{\infty})^N \mathds{1}_{\overline{V}^{-1} ((K, K))} (h). 
\end{align} 

To estimate the norm $\| W_{\psi} \psi \|_{W^R (L^{\infty}, Y)}$, 
recall that since $v : G \to (0,\infty)$ is admissible, 
there exists an $s \in (0,\infty)$ and a locally bounded weight $v_0 : H \to (0,\infty)$
yielding the estimate $v(x,h) \leq (1+|x|)^s v_0 (h)$.  
Using this estimate and \eqref{eq:local_estimate} directly entails
\begin{align*}
\| W_{\psi} \psi \|_{W^R (L^{\infty}, Y)} 
&\leq C'  \bigg\| |\det (h^{-1}) |^{\frac{1}{q}} v_0 (h) \mathds{1}_{\overline{V}^{-1} ((K, K))} (h) (1+\|h\|_{\infty})^N \bigg\| (1+|x|)^{(s - N)} \bigg\|_{L^p } \bigg\|_{L^q} \\
&\leq C' \sup_{h \in \overline{V}^{-1} ((K, K))} \bigg( |\det (h^{-1})|^{\frac{1}{q}} v_0 (h) \bigg) 
\bigg\| (1+|x|)^{(s - N)} \bigg\|_{L^p} \bigg\| \mathds{1}_{ \overline{V}^{-1} ((K,K))} \bigg\|_{L^q }
\end{align*}
for arbitrary $N \in \mathbb{N}$. 
Hence, choosing $N \in \mathbb{N}$ such that
$(1 + |x|)^{(s-N)} \in L^p (\RR^d)$, that is, $N > d/p + s$,
 yields the claim. 
 
The $v$-integrability follows since $W^R (L^{\infty}, L^1_v)(G) \hookrightarrow L^1_v (G)$. 
\end{proof}

\section{Coorbit spaces associated to integrably admissible dilation groups}
\label{sec:coorbit}
This section considers coorbit spaces associated with integrably admissible dilation groups. 
In the first subsection the spaces are formally defined and several of their basic properties are proven. 
Section \ref{sec:discretization} is devoted to several discretization results. 

\subsection{Definition and basic properties}

\begin{definition} \label{def:coorbit}
Let $H \leq \mathrm{GL}(d, \mathbb{R})$ be integrably admissible 
with essential frequency support $\mathcal{O}$.
Let $v : G \to (0, \infty)$ be an admissible weight 
and let $Y = L^{p,q}_v (G)$ for $p, q \in [1,\infty]$. 
For a fixed admissible vector $\psi \in \mathcal{S}_{\mathcal{O}} (\mathbb{R}^d)$,
the associated \emph{coorbit space} is defined by
\begin{align} \label{eq:definition_coorbit}
 \CoSYpsi = \bigg \{ f \in \SOdual
: W_{\psi} f \in Y \bigg\}
\end{align}
and will be equipped with the norm $\| f \|_{\CoSYpsi} = \| W_{\psi} f \|_Y$. 
\end{definition}

In settling the basic properties of coorbit spaces, 
we first prove that the subspace
\[
Y_{\psi} := \{ F \in Y \; : \; F = F \ast W_{\psi} \psi \}
\]
is a reproducing kernel Banach space in $Y$,
provided that $\psi \in L^2 (\RR^d)$
is an analyzing vector.

\begin{lemma} \label{lem:reproducing}
Let $H \leq \mathrm{GL}(d, \mathbb{R})$ be integrably admissible 
with essential frequency support $\mathcal{O}$. 
Let $v : G \to (0, \infty)$ be an admissible weight 
and let $Y = L^{p,q}_v (G)$ for $p, q \in [1,\infty]$. 
Then, for $\psi \in \mathcal{S}_{\mathcal{O}} (\RR^d)$,
the space $Y_{\psi}$ is closed in $Y$. 
\end{lemma}
\begin{proof}
Let $(F_n)_{n \in \mathbb{N}}$ be a sequence in $Y_{\psi}$
converging to some $F \in Y$. Then there exists a subsequence
$(F_{n_k})_{k \in \mathbb{N}}$ converging to $F \in Y$ $\mu_G$-almost everywhere.
Since $W_{\psi} \psi \in L^{p', q'}_{1/v} (G)$ by Lemma \ref{lem:coefficient_estimate}(i), 
it follows from Proposition \ref{prop:admissible_weight_properties}(iii) that 
$(F \ast W_{\psi} \psi) (g)$ is well-defined for every $g \in G$. 
Hence, for $\mu_G$-a.e. $g \in G$, 
\begin{align*}
&|F (g) - (F \ast W_{\psi} \psi)(g) | \\
&\quad \quad \leq |F(g) - F_{n_k} (g)| 
+ |F_{n_k} (g) - (F_{n_k} \ast W_{\psi} \psi) (g) |
+ |(F_{n_k} \ast W_{\psi} \psi) (g) - (F \ast W_{\psi} \psi) (g) | \\
& \quad \quad \leq |F(g) - F_{n_k} (g) | + \| W_{\psi} \psi \|_{L^{p',q'}_{1/v}} \|F_{n_k} - F \|_{L^{p,q}_v},
\end{align*}
since $F_{n_k} - (F_{n_k} \ast W_{\psi} \psi)  = 0$ $\mu_G$-a.e. on $G$. 
From this it follows easily that $F = F \ast W_{\psi} \psi$ $\mu_G$-a.e. on $G$, 
and thus $F \in Y_{\psi}$, as required. 
\end{proof}

\begin{proposition} \label{prop:coorbit_basic}
Let $\CoSYpsi$ be a coorbit space defined by 
an admissible vector $\psi \in \mathcal{S}_{\mathcal{O}} (\RR^d)$.
Then the following assertions hold:
\begin{enumerate}[(i)]
\item The space $\CoSYpsi$ is a $\pi$-invariant Banach space. 
\item The space $\CoSYpsi$ is independent of the chosen admissible vector,
 i.e., for any two
admissible vectors $\psi_1, \psi_2 \in \mathcal{S}_{\mathcal{O}} (\mathbb{R}^d)$,
it holds $\CoSYone = \CoSYtwo$, with equivalent norms. 
\item The map $W_{\psi} : \CoSYpsi \to Y_{\psi}$ is an isometric isomorphism. 
\end{enumerate}
\end{proposition}
\begin{proof}
(i). The subadditivity and absolute homogeneity of $\| \cdot \|_{\CoSYpsi}$ 
are clear. To show that $\| \cdot \|_{\CoSY}$ is positive definite, let $f \in \CoSYpsi$ and
suppose that $\| f \|_{\CoSYpsi} = 0$. 
Then $\langle f, \pi(x,h) \psi \rangle = 0$ for $\mu_G$-a.e. $(x,h) \in G$. 
Using this, together with the identity \eqref{eq:weak_continuity}, yields that
\[
\langle f, \phi \rangle = \int_G \langle f, \pi(g) \psi \rangle 
	\langle \pi(g) \psi, \phi \rangle \; d\mu_{G} (g) = 0
\]
for arbitrary $\phi \in \mathcal{S}_{\mathcal{O}} (\RR^d)$. 
Thus $f = 0$, showing that $\| \cdot \|_{\CoSYpsi}$ is indeed a norm. 

For the completeness of $\CoSYpsi$, 
let $(f_n)_{n \in \mathbb{N}}$ be a Cauchy sequence in $\CoSYpsi$. 
Then $(W_{\psi} f_n)_{n \in \mathbb{N}}$ is a Cauchy sequence in 
$Y_{\psi}$, and hence converges to some $F \in Y_{\psi}$ by Lemma \ref{lem:reproducing}.
 The map
$W_{\psi} : \mathcal{S}_{\mathcal{O}} (\mathbb{R}^d) \to L^{p',q'}_{1/v} (G)$  
is well-defined by Lemma \ref{lem:coefficient_estimate}(i). Combining this with 
Proposition \ref{prop:admissible_weight_properties}(iii)
yields that 
\[ f : \mathcal{S}_{\mathcal{O}} (\mathbb{R}^d) \to \mathbb{C}, 
\quad \phi \mapsto 
\int_G F(g) \langle \pi(g) \psi, \phi \rangle \; d\mu_G (g)
\]
is a well-defined element of $\SOdual$. 
Evaluating $f \in \SOdual$ on $\phi = \pi(y) \psi$, 
$y \in G$, gives
\begin{align*}
W_{\psi} f (y) = \int_G F(g) \langle \pi (g) \psi, \pi(y) \psi \rangle \; d\mu_G (g) 
= F \ast W_{\psi} \psi (y) = F(y). 
\end{align*}
From this it easily follows that $\CoSYpsi$ is complete, and thus a Banach space.  
The $\pi$-invariance of $\CoSYpsi$ follows directly from the left translation-invariance of $Y$. 

(ii). Let $\psi_1, \psi_2 \in \mathcal{S}_{\mathcal{O}} (\mathbb{R}^d)$ be admissible.
Then $W_{\psi_2} \psi_1 \in L^{1}_{w} (G)$ for any admissible control weight $w$ 
by Proposition \ref{prop:amalgam_estimate}. 
Thus,  if $W_{\psi_1} f \in Y$, 
then 
$W_{\psi_2} f =W_{\psi_1} f \ast W_{\psi_2} \psi_1 \in Y \ast L^1_w (G) \hookrightarrow Y$,
with
\[
\| W_{\psi_2} f \|_{Y} = \| W_{\psi_1} f \ast W_{\psi_2} \psi_1 \|_Y \leq \| W_{\psi_2} \psi_2 \|_{L^1_w} \| W_{\psi_1} f \|_{Y}. 
\]
by Lemma \ref{prop:admissible_weight_properties}(ii). 
Reversing the role of $\psi_1$ and $\psi_2$ yields $\| W_{\psi_1} f \|_Y \asymp \| W_{\psi_2} f \|_Y$, 
as required.  

(iii). We prove the surjectivity of $W_{\psi} : \CoSYpsi \to Y_{\psi}$. 
The inclusion $W_{\psi} (\CoSYpsi) \subseteq Y_{\psi}$ follows by 
Lemma \ref{lem:technical_lemma1}(ii). 
For the reverse, let $F \in Y_{\psi}$.
Then similar arguments as in (i) show that the map
$
\phi \mapsto \int_G F(g) \langle \pi(g) \psi, \phi \rangle \; d\mu_G (g)
$
defines an element $f \in \SOdual$ satisfying
$W_{\psi} f = F \ast W_{\psi} \psi = F \in Y$, and thus $f \in \CoSYpsi$.  
\end{proof}

In light of the above independence result, 
we will omit the dependence of $\CoSYpsi$ on $\psi \in \mathcal{S}_{\mathcal{O}} (\RR^d)$ 
and simply write $\CoSY = \CoSYpsi$ in the remainder. 

\subsection{Discretizations of the reproducing formula}  \label{sec:discretization}
In this section the results on the discretization of convolution operators as developed in 
\cite{feichtinger1989banach1, grochenig1991describing} will be applied to the current setting. 
For this, we adopt the following terminology. 

\begin{definition}
Let $G$ be a locally compact group, let $X = (x_i)_{i \in I} \subseteq G$ be a countable, discrete set 
and let $U \subseteq G$ be a unit neighborhood. 
\begin{enumerate}[(i)]
\item The set $X = (x_i)_{i \in I}$  is called \emph{$U$-dense} in $G$ if $\bigcup_{i \in I} x_i U = G$.
\item The set $X = (x_i)_{i \in I}$ is called \emph{$U$-separated} in $G$ if $(x_i U)_{i \in I}$ is 
pairwise disjoint. The set $X$ is called \emph{separated} if it is $U$-separated
for some unit neighborhood $U \subseteq G$ and is called \emph{well-spread} if it is $U$-dense 
and separated. 
\item The set $X = (x_i)_{i \in I}$ is \emph{relatively separated} if it is a finite union of
separated sets. 
\end{enumerate}
\end{definition}

\begin{remark}
In $G = \mathbb{R}^d \rtimes H$, the typical discrete sets $X \subseteq G$ 
possess the form
\[
X = \{ (h_j x_k, h_j ) \; | \; j \in J, k \in I \}
\]
for  $U$-dense and separated sets $(h_j)_{j \in J} \subseteq H$ and 
$(x_k)_{k \in I} \subseteq \RR^d$. 
Any such set $X$ can be shown to be uniformly dense and separated in $G$. 
\end{remark}

\begin{definition}
Let $Y = L^{p,q}_v (G)$ for $p,q \in [1,\infty]$ and some weight $v : G \to \mathbb{R}^+$. 
Let $X = (x_i)_{i \in I} \subseteq G$ 
be a relatively separated set. For a relatively compact neighborhood of the identity
$U \subseteq G$, define
\[
\big\| (c_i)_{i \in I} \big\|_{Y_d (X)} := \bigg\| \sum_{i \in I} |c_i| \; \mathds{1}_{x_i U} \bigg\|_Y
\]
and $Y_d (X) = \{ (c_i)_{i \in I} \in \CC^I \; : \; \|(c_i)_{i \in I} \|_{Y_d (X)} < \infty \}$. 
\end{definition}

For further reference, we remark some basic properties of $Y_d (X)$ to be used in the sequel. 

\begin{remark}
\begin{enumerate}[(i)]
\item The sequence space $Y_d (X)$ forms a well-defined Banach space, 
independent of the choice of the defining neighborhood $U$. 
\item The norm $\| \cdot \|_{Y_d (X)}$ is equivalent to a weighted $\ell^{p,q}$-norm. 
In particular, for the Lebesgue space $Y = L^p (G)$, the associated sequence space is $Y_d (X) = \ell^p (X)$
for all $p \in [1,\infty]$. 
\end{enumerate}
\end{remark}

The following result on atomic decompositions and frames is obtained by applying the
general discretization procedure of convolution operators in
\cite[Section 5]{feichtinger1989banach1} and \cite[Section 4]{grochenig1991describing} to the integrable  
kernel $ W_{\psi} \psi$.
The proofs are similar to 
\cite[Theorem S]{grochenig1991describing} and \cite[Theorem T]{grochenig1991describing}.

\begin{theorem} \label{thm:atomic}
Suppose that 
$\psi \in \mathcal{S}_{\mathcal{O}} (\RR^d)$ is admissible. 
Then there exists a unit neighborhood $U \subset G$ such that,
 for any $U$-dense and relatively separated set $X = (x_i)_{i \in I}$,
the following assertions hold:
\begin{enumerate}[(i)]
\item Given $f \in \CoSY$, there exists $(c_i (f))_{i \in I} \in Y_d (X)$ such that
\begin{align} \label{eq:atomic}
f = \sum_{i \in I} c_i (f) \pi(x_i) \psi, 
\end{align}
and
$\| (c_i (f))_{i \in I} \|_{Y_d} \lesssim \| f \|_{\CoSY}$, 
with an implicit constant only depending on $\psi \in \mathcal{S}_{\mathcal{O}} (\RR^d)$. 
\item Every $(c_i)_{i \in I} \in Y_d (X)$ defines an element 
\begin{align} \label{eq:reverse_atomic}
f = \sum_{i \in I} c_i \pi(x_i) \psi 
\end{align}
in $\CoSY$ with $\| f \|_{\CoSY} \lesssim \|(c_i)_{i \in I} \|_{Y_d}$. 
\item The norm equivalence
\begin{align} \label{eq:banach_frame}
 \| f \|_{\CoSY} \asymp \big\| ( \langle f, \pi(x_i) \psi \rangle)_{i \in I} \|_{Y_d} 
\end{align}
holds for all $f \in \CoSY$. 
\end{enumerate}
The series \eqref{eq:atomic} and \eqref{eq:reverse_atomic} converge unconditionally 
in the norm of $\CoSY$ if $p,q \in [1,\infty)$ and in the weak$^*$-topology, otherwise. 
\end{theorem}
\begin{proof}
By assumption, we have $W_{\psi} \psi \in W^R (L^{\infty}, L^1_w)$ for a control weight $w : G \to \mathbb{R}^+$.

(i) An application of  \cite[Theorem 4.10]{grochenig1991describing} yields the invertibility
of the operator 
\[ 
T_{\Phi} : Y_{\psi} \to Y_{\psi}, \quad F \mapsto \sum_{i \in I} \langle \phi_i, F \rangle L_{x_i} W_{\psi} \psi
\]
for some partition of unity $\Phi = (\phi_i)_{i \in I}$ subordinate to $U \subseteq G$. 
Applying this to $F = W_{\psi} f \in Y_{\psi}$ for $f \in \CoSY$ yields that
\[
F = T_{\Phi} T_{\Phi}^{-1} F = \sum_{i \in I} \langle \phi_i, T_{\Phi}^{-1} F \rangle L_{x_i} W_{\psi} \psi. 
\]
Pulling this back to $\CoSY$ through 
$W^{-1}_{\psi} : Y_{\psi} \to \CoSY$ gives
$
f = \sum_{i \in I} \langle \phi_i, T_{\Phi}^{-1} W_{\psi} f\rangle \pi (x_i) \psi
$. 
The fact that $(\langle \phi_i, T_{\Phi}^{-1} W_{\psi} f \rangle)_{i \in I} \in Y_d (X)$ 
follows directly from \cite[Proposition 5.1]{feichtinger1989banach1}. 

(ii) By \cite[Lemma 5.2]{feichtinger1989banach1}, the mapping 
\begin{align} \label{eq:synthesis}
(c_i)_{i \in I} \mapsto \sum_{i \in I} c_i L_{x_i} W_{\psi} \psi 
\end{align}
is well-defined and bounded from $Y_d (X)$ into $Y$, 
with convergence in $Y$ if $p,q \in [1,\infty)$ and point-wise, otherwise. 
Let $(c_i)_{i \in I} \in Y_d (X) \hookrightarrow \ell^{\infty} (I)$ be arbitrary. 
Define $F \in L^{\infty} (G)$ by
\[
F(x) = W_{\psi} \bigg( \sum_{i \in I} c_i \pi(x_i) \psi \bigg) (x) = \sum_{i \in I} c_i L_{x_i} W_{\psi} \psi (x)
\] 
for $x \in G$. Choose an enumeration $(i_n)_{n \in \mathbb{N}}$ of the index set $I$ and
define $f_n \in \mathcal{S}_{\mathcal{O}} (\mathbb{R}^d)$ 
as
\begin{align} \label{eq:fn_density}
f_n := \sum_{\ell = 1}^n c_{i_{\ell}} \pi (x_{i_{\ell}}) \psi. 
\end{align}
Then $W_{\psi} f_n (x) = \sum_{\ell = 1}^n c_{i_{\ell}} L_{x_{i_{\ell}}} W_{\psi} \psi (x)$, 
and the triangle inequality and solidity of $Y$ show  that $f_n \in \CoSY$.  
By the Lebesgue dominated convergence theorem, 
\begin{align*}
\lim_{n \to \infty} \langle f_n, \phi \rangle 
&=  \int_G \lim_{n \to \infty} W_{\psi} f_n (x) \langle \pi(x,h) \psi, \phi \rangle \; d\mu_G (x,h) 
= \int_G F(x) \langle \pi(x,h) \psi, \phi \rangle \; d\mu_G (x,h) \\
&= \langle f, \phi \rangle
\end{align*}
for all $\phi \in \mathcal{S}_{\mathcal{O}} (\mathbb{R}^d)$, 
where $f := \sum_{i \in I} c_i \pi (x_i) \psi \in \SOdual$.  
Thus $f_n \to f$ in the $w^*$-topology of $\SOdual$ as $n \to \infty$.
Moreover, if $p, q \in [1,\infty)$, then the norm convergence of the series in \eqref{eq:synthesis}
yields that $f_n \to f$ in $\CoSY$ as $n \to \infty$. 

(iii)  An application of  \cite[Theorem 4.11]{grochenig1991describing} yields the invertibility
of the operator 
\[ 
S_{\Phi} : Y_{\psi} \to Y_{\psi}, \quad F \mapsto \sum_{i \in I} F(x_i) \phi_i \ast W_{\psi} \psi
\]
for a suitable partition of unity $\Phi = (\phi_i)_{i \in I}$ subordinate to $U \subseteq G$. 
In particular, for $F = W_{\psi} f$ with $f \in \CoSY$, we obtain  
\[
F = S_{\Phi}^{-1} S_{\Phi} F = S^{-1}_{\Phi} \bigg( 
\sum_{i \in I} F(x_i) \phi_i \ast W_{\psi} \psi \bigg)
\]
and thus
$
f = W_{\psi}^{-1} S_{\Phi}^{-1} \big(\sum_{i \in I} \langle f, \pi(x_i) \psi \rangle \phi_i \ast W_{\psi} \psi \big).
$
The norm equivalence \eqref{eq:banach_frame} follows from this by a direct calculation:
\begin{align*}
\| f \|_{\CoSY} &\leq \| S_{\Phi}^{-1} \|_{Y_{\psi} \to Y_{\psi}} \bigg\| \sum_{i \in I} F(x_i) \phi_i \ast W_{\psi} \psi \bigg\|_{Y} \\
&\leq \| S_{\Phi}^{-1} \|_{Y_{\psi} \to Y_{\psi}} \bigg\| \sum_{i \in I} F(x_i) \phi_i \bigg\|_Y \| W_{\psi} \psi \big\|_{L^1_w} \\ 
&\leq \| S_{\Phi}^{-1} \|_{Y_{\psi} \to Y_{\psi}} \big\| \big( F(x_i) \big)_{i \in I} \big\|_{Y_d (X)} \| W_{\psi} \psi \big\|_{L^1_w} \\
&\lesssim \| S_{\Phi}^{-1} \|_{Y_{\psi} \to Y_{\psi}} \big\|f \big\|_{\CoSY} \| W_{\psi} \psi \big\|_{L^1_w},
\end{align*}
where the last line follows from the boundedness of 
$\CoSY \ni f \mapsto (W_{\psi} f (x_i))_{i \in I} \in Y_d (X)$, 
with the implicit constant independent of $f$, e.g., see \cite[Theorem 8.1]{feichtinger1989banach2}. 
\end{proof}

\begin{corollary} \label{cor:density_coorbit}
Let $H \leq \mathrm{GL}(d, \mathbb{R})$ be integrably admissible 
with essential frequency support $\mathcal{O}$. The space $\mathcal{S}_{\mathcal{O}} (\RR^d)$
is norm dense in $\CoSY$ if $p, q \in [1,\infty)$ and $w^*$-dense, otherwise. 
Moreover, for any $f \in \CoSY$, there exists a sequence $(f_n)_{n \in \mathbb{N}}$ in $\mathcal{S}_{\mathcal{O}} (\RR^d)$ with $f_n \in \mathcal{S}_{\mathcal{O}} (\RR^d)$ 
such that $f_n \to f$ in $\SOdual$ 
and a constant $C>0$ such that $\| f_n \|_{\CoSY} \leq C \| f \|_{\CoSY}$ 
for all $n \in \mathbb{N}$. 
\end{corollary}
\begin{proof}
Let $f \in \CoSY$ and $\psi \in \mathcal{S}_{\mathcal{O}} (\RR^d)$.
Then $f = \sum_{i \in I} c_i (f) \pi (x_i) \psi$ for some coefficients $(c_i (f))_{i \in I} \in Y_d (X)$
 by Theorem \ref{thm:atomic}. Let $f_n \in \mathcal{S}_{\mathcal{O}} (\RR^d)$ be as in \eqref{eq:fn_density}. 
 Then  $f_n \to f$ as $n \to \infty$
  in the norm topology of $\CoSY$ if $p,q \in [1,\infty)$ and in the $w^*$-topology, otherwise. 
  
Lastly, by Theorem \ref{thm:atomic}, 
there exist constants $C_1, C_2 > 0$ such that
\[
\| f_n \|_{\CoSY} \leq C_2 \| (c_i (f))_{i \in I} \mathds{1}_{\{i_1, ..., i_n \}} \|_{Y_d} 
\leq C_1 C_2 \| f \|_{\CoSY},
\]
which yields the claim. 
\end{proof}

For other combined results on frames and atomic decompositions, 
the interested reader is referred to \cite[Section 5]{grochenig1991describing}. 
As the involved discretization techniques only rely on the integrability of the representation 
and the existence of a reproducing formula \cite[Remark 6.6]{grochenig1991describing}, 
the results could also be obtained in the setting of the current paper, with similar proofs.

Dual to the results on frames and atomic decompositions, 
we show the validity of the interpolation property and the existence of $p$-Riesz sequences.
The following interpolation property of the wavelet transform is \cite[Theorem 8.2]{feichtinger1989banach2}, 
adapted to our current setting.

\begin{theorem} \label{prop:interpolation}
Suppose $\psi \in \mathcal{S}_{\mathcal{O}} (\mathbb{R}^d)$ is normalized. 
Then there exists a compact set $K \subseteq G$ such that,
for every $K$-separated set $X = (x_i)_{i \in I}$ in $G$, 
the following assertions hold:
\begin{enumerate}[(i)]
\item Given $c = (c_i)_{i \in I} \in Y_d (X)$,
there exists $f \in \CoSY$ such that $W_{\psi} f (x_i) = c_i$
for all $i \in I$, and  $\| f \|_{\CoSY} \lesssim \|c \|_{Y_d(X)}$, with an implicit 
constant only depending on $\psi \in \mathcal{S}_{\mathcal{O}} (\mathbb{R}^d)$. 
\item The norm equivalence 
\[
 \|c\|_{Y_d (X)} 
\asymp \bigg\| \sum_{i \in I} c_i \pi (x_i) \psi \bigg\|_{\CoSY}
\]
holds for all $c = (c_i)_{i \in I} \in Y_d (X)$. 
\end{enumerate}
\end{theorem} 
\begin{proof}
(i) By \cite[Theorem 7.2]{feichtinger1989banach2}, 
there exists a compact set $K \subseteq G$
such that, for every $K$-separated set $X = (x_i)_{i \in I}$, the mapping
$F \mapsto (F \ast W_{\psi} \psi (x_i))_{i \in I}$ is a bounded surjection from $Y$ onto $Y_d (X)$,
i.e., given $c = (c_i)_{i \in I} \in Y_d (X)$, 
there exists $F \in Y$ such that $(F \ast W_{\psi} \psi)(x_i) = c_i$ 
for all $i \in I$, and $\|F\|_Y \lesssim \|c\|_{Y_d (X)}$. 
Since $W_{\psi} : \CoSY \to Y_{\psi}$ is an isometric isomorphism,
the claim follows readily. 

(ii)
Let $c = (c_i)_{i \in I} \in Y_d (X)$ be arbitrary. 
Choose a $c' = (c'_i)_{i \in I} \in Y'_d (X)$ satisfying $\| c' \|_{Y'_d (X)} = 1$
and $\langle c', c \rangle = \sum_{i \in I} c'_i \overline{c_i} = \|c\|_{Y_d (X)}$. 
By part (i), there exists an $f \in \CoSYdual$ such that 
$\|f\|_{\CoSYdual} \lesssim 1$ and $W_{\psi} f (x_i) = c'_i$ for all $i \in I$. 
Hence
\begin{align*}
\|c \|_{Y_d (X)} &= \sum_{i \in I} c'_i \overline{c_i} 
= \bigg \langle f, \sum_{i \in I} c_i \pi (x_i) \psi \bigg \rangle 
\leq \|f\|_{\CoSYdual} \bigg\| \sum_{i \in I} c_i \pi(x_i) \psi \bigg\|_{\CoSY} \\
&\lesssim \bigg \| \sum_{i \in I} c_i \pi(x_i) \psi \bigg\|_{\CoSY}.
\end{align*}
The upper bound follows  by Theorem \ref{thm:atomic}. 
\end{proof}

\section{Besov-type decomposition spaces associated to induced covers}
\label{sec:decomposition}
This section is devoted to decomposition spaces and the main ingredients 
for defining these spaces, namely admissible coverings and associated
subordinate partitions of unity. In particular, we will construct
covers and partitions of unity that are suitable for identifying coorbit spaces 
as decomposition spaces. 

\subsection{Admissible and induced covers}

We  allow for admissible covers of a rather general type \cite{feichtinger1985banach}. 

\begin{definition}\label{def:AdmissibleCovering}
A family $\mathcal{Q} = (Q_i)_{i \in I}$ of non-empty subsets $Q_i \subset \mathcal{O}$ is
called an \emph{admissible covering} of the open subset $\mathcal{O} \subset \RHat$, if
\begin{enumerate}
  \item[(i)] $\mathcal{Q}$ is a covering of $\mathcal{O}$, i.e,
             $\mathcal{O} = \bigcup_{i \in I} Q_i$;

  \item[(ii)] $N_{\mathcal{Q}} := \sup_{i \in I} |i^*| < \infty$,
 where $i^* := \{ \ell \in I \; : \; Q_{\ell} \cap Q_i \neq \emptyset \}$.
\end{enumerate}
\end{definition}

The aim of this subsection is to construct a cover
of the essential frequency support $\mathcal{O}$ 
that is induced by the action of the associated 
integrably admissible dilation group $H \leq \mathrm{GL}(d, \mathbb{R})$.  
The construction relies on the following technical lemmata,
which extend corresponding results in \cite{fuehr2015wavelet} to the setting of 
integrably admissible dilation groups. 

\begin{lemma} \label{lem:neighborhoods_bounded}
Let $K_1, K_2 \subseteq \mathcal{O}$ be compact. Given a relatively compact unit neighborhood $V \subseteq H$ and a $V$-separated family
$(h_i)_{i \in I} \subseteq H$, there exists a constant $C = C(K_1, K_2,V) > 0$ such that, for all $h \in H$,
\[
\# I_h (K_1, K_2) \leq C,
\]
 where $I_h(K_1, K_2) := \{ i \in I \; | \; h^{-T} K_1 \cap h_i^{-T} K_2 \neq \emptyset \} $.
\end{lemma}
\begin{proof}
Let $K_1, K_2 \subseteq \mathcal{O}$ be compact and let $(h_i)_{i \in I} \subseteq H$ 
be a $V$-separated set. Set $L := ((K_1,K_2))$, which is relatively compact by Corollary \ref{cor:KKcp}. 
Fix $h \in H$. If $i \in I_h(K_1, K_2)$, then $h^{-T} K_1 \cap h_i^{-T} K_2 \not= \emptyset$, 
and thus $h_i \in h  L$. Hence 
$h_i V^{\circ} \subseteq hL\overline{V}$, and $(h_i V^{\circ})_{i \in I_h}$ is a pairwise disjoint collection
of subsets of $hL \overline{V}$. Therefore, for a finite subset $J \subseteq I_h(K_1, K_2)$, 
\[
\#J = \frac{1}{\mu_H (V^{\circ})} \sum_{i \in J} \mu_H (h_i V^{\circ}) 
= \frac{1}{\mu_H (V^{\circ})} \mu_H \bigg( \bigcup_{i \in J} h_i V^{\circ} \bigg)
\leq \frac{\mu_H (hL\overline{V})}{\mu_H (V^{\circ})} 
= \frac{\mu_H (L\overline{V})}{\mu_H (V^{\circ})}. 
\]
Since $J \subseteq I_h$ can be chosen arbitrary, the result follows. 
\end{proof}

\begin{lemma} \label{lem:existence_covering}
Suppose that $(h_i)_{i \in I}$ is a well-spread family in $H$. 
Then there exists a relatively compact set $Q \subset \RHat$ 
satisfying $\overline{Q} \subset \mathcal{O}$ and $\mathcal{O} = \bigcup_{i \in I} h_i^{-T} Q$. 
\end{lemma}
\begin{proof}
Let $\mathcal{O} = H^T C$ be the essential frequency support of $H$. 
Let $(h_i)_{i \in I}$ be a well-spread family and let $U \subseteq H$ be a relatively compact set
such that $H = \bigcup_{i \in I} h_i U$. Then
\[
\mathcal{O} = H^T C = H^{-T} C = \bigcup_{i \in I} h_i^{-T} U^{-T} C, 
\]
which shows that $\mathcal{Q} = (h_i^{-T} U^{-T} C)_{i \in I}$ is a covering of $\mathcal{O}$. 
Note that $Q := U^{-T} C $ has a compact closure 
$\overline{Q} = \overline{U}^{-T} C \subseteq \mathcal{O}$ since $C \subseteq \mathcal{O}$ is compact. 
\end{proof}

Motivated by the previous result, we make the following formal definition. 

\begin{definition} \label{def:induced_cover}
Let $(h_i)_{i \in I} \subseteq H$ be well-spread and let $Q \subseteq \RHat^d$ be  relatively compact with
 $\overline{Q} \subseteq \mathcal{O}$. A covering $\mathcal{Q} = (Q_i)_{i \in I}$ of $\mathcal{O}$ possesing the form 
$Q_i = h_i^{-T} Q$ is called a \emph{covering of $\mathcal{O}$ induced by $H$}. 
\end{definition}

Any induced cover is admissible in the sense of Definition \ref{def:AdmissibleCovering}:

\begin{proposition} \label{prop:induced_admissible}
Any cover $\mathcal{Q} = (Q_i)_{i \in I}$ of $\mathcal{O}$ induced by $H$ is an 
admissible cover.
\end{proposition}
\begin{proof}
Let  $\mathcal{Q} = (h_i^{-T} Q)_{i \in I}$ be an induced cover of $\mathcal{O}$, 
with $Q \subseteq \RHat^d$ having compact closure $\overline{Q} \subseteq \mathcal{O}$.
Then, by Lemma \ref{lem:neighborhoods_bounded}, there exists a constant $C > 0$
such that $|I_{h_i} | \leq C$ for all $i \in I$, where 
$ I_{h_i} := \{j \in I \; | \; h_j^{-T} \overline{Q} \cap h_i^{-T} \overline{Q} \neq \emptyset \}. $
But $i^* = I_{h_i}$, and thus $\mathcal{Q}$ is admissible. 
\end{proof}

\subsection{Weights, discretizations and pull-backs} 

\begin{definition}
Let $\mathcal{Q} = (Q_i)_{i \in I}$ be an admissible covering of an open set $\mathcal{O} \subset \RHat^d$. 
\begin{enumerate}[(i)]
\item A function $u : \mathcal{O} \to \mathbb{R}^+$ is called \emph{$\mathcal{Q}$-moderate}
 if there exists a $C_{\mathcal{Q}} > 0$
such that $u(x) / u(y) \leq C_{\mathcal{Q}}$ for all $x, y \in Q_i$ and $i \in I$. 
\item A sequence $u : I \to \mathbb{R}^+$ is called \emph{$\mathcal{Q}$-moderate} if
$ \sup_{i \in I} \sup_{\ell \in i^*} u_i / u_{\ell} < \infty.$
\item A sequence $u' : I \to \mathbb{R}^+$ is called a \emph{discretization} of 
$u : \mathcal{O} \to \mathbb{R}^+$ if, for every $i \in I$, there exists an $x_i \in Q_i$
such that $u'_i = u(x_i)$. 
\end{enumerate}
\end{definition}

\begin{remark} \label{rem:discrete_weights}
A discretization of a $\mathcal{Q}$-moderate function is $\mathcal{Q}$-moderate as a sequence and, 
moreover, any two discretizations are equivalent as weighting functions. 
\end{remark}

In the sequel, we will repeatedly transfer a weight $v : H \to \mathbb{R}^+$ 
into a function $u : \mathcal{O} \to \mathbb{R}^+$ by means of a pull-back 
through a cross-section. More precisely, if 
$\sigma : \mathcal{O} \to H$ is a cross-section, 
then $\sigma^* v = u$, with $\sigma^* v (\xi) = v(\sigma(\xi))$ for $\xi \in \mathcal{O}$. 
We formally define the following. 

\begin{definition}
Let $\mathcal{O} = H^T C$ be an essential fequency support of $H \leq \mathrm{GL}(d, \mathbb{R})$ 
and let $v : H \to \mathbb{R}^+$ be a weighting function. 
A function $u : \mathcal{O} \to \mathbb{R}^+$ is called a \emph{pull-back of $v$ from $H$ onto $\mathcal{O}$} 
if, for each $\xi \in \mathcal{O}$, there exist $h_{\xi} \in H$ and $\xi_0 \in C$ such that 
$\xi = h^{-T}_{\xi} \xi_0$ and $u(\xi) = v(h_{\xi})$. 
\end{definition}

By suitable assumptions on the weight $v$, a pull-back $u$ is $\mathcal{Q}$-moderate 
and any two pull-backs are equivalent:

\begin{lemma} \label{lem:pull-back}
Let $\mathcal{O} = H^T C$ be an essential frequency support of $H \leq \mathrm{GL}(d, \mathbb{R})$ 
and let $v : H \to \mathbb{R}^+$ be a submultiplicative weight. 
\begin{enumerate}[(i)]
\item If
$
u : \mathcal{O} \to \mathbb{R}^+ 
$ 
is a pull-back of $v$ from $H$ onto $\mathcal{O}$ and if $Q \subset \mathcal{O}$ is compact, 
then there exists a constant $C_0 = C_0(v, C, Q) > 0$ 
such that
\begin{align} \label{eq:Q-moderate}
C_0^{-1} v(h) \leq u(\xi) \leq C_0 v(h) 
\end{align}
for all $\xi \in \mathcal{O}$ and $h \in H$ satisfying $\xi \in h^{-T} Q$. 
\item Any pull-back $u : \mathcal{O} \to \mathbb{R}^+$ of $v$ is $\mathcal{Q}$-moderate for any induced cover $\mathcal{Q}$ of $\mathcal{O}$. 
\item Any two pull-backs $u_1, u_2 : \mathcal{O} \to \mathbb{R}^+$ are equivalent as weights. 
\end{enumerate}
\end{lemma}
\begin{proof}
(i) Let $L := ((C, Q))$, which is relatively compact in $H$ by Corollary \ref{cor:KKcp}. 
Fix  $h \in H$ and $\xi \in \mathcal{O}$ satisfying $\xi \in h^{-T} Q$. 
Choose $h_{\xi} \in H$ and $\xi_0 \in C$ such that $u(\xi) = v(h_{\xi})$ and $h^{-T}_{\xi} \xi_0 = \xi \in h^{-T} Q$. Then 
\[
\emptyset \neq h_{\xi}^{-T} \{\xi_0 \} \cap h^{-T} Q \subseteq h^{-T}_{\xi} C \cap h^{-T} Q,
\]
and thus $h \in h_{\xi} L$. Then, since $v : H \to \mathbb{R}^+$ 
is submultiplicative, there exists a constant $C_0 = C_0 (v, L) > 0$ such that
\[
C_0^{-1} v(h' k) \leq v(h_{\xi}) \leq C_0 v(h' k)
\]
for all $h' \in H$ and $k \in L$. Since $h = h_{\xi} k$ for some $k \in L$, this shows \eqref{eq:Q-moderate}. 

(ii) Suppose that $\mathcal{Q} = (h_i^{-T} Q)_{i \in I}$ is an induced cover and let $\xi, \omega \in Q_i = h_i^{-T} Q$  for some $i \in I$. Then \eqref{eq:Q-moderate} yields a constant $C_0 (v, C, Q) > 0$ such that
\[
u(\xi) \leq C_0 v(h_i) \leq C_0^2 u(\omega)
\]
for all $i \in I$. By symmetry, it follows that $u$ is $\mathcal{Q}$-moderate. 

(iii) Suppose that $u_1, u_2$ are two pull-backs. Let $\xi \in \mathcal{O} = H^T C$ 
be arbitrary. Choose $h \in H$ and $\xi_0 \in C$ such that $\xi = h^{-T} \xi_0 \in h^{-T} C$. 
By \eqref{eq:Q-moderate}, there exist constants $C_1 = C_1(u_1, C) > 0$ and $C_2 = C_2(u_2, C) > 0$ such that
\[
C_1^{-1} C_2^{-1} u_2 (\xi) \leq C_1^{-1} v(h) \leq u_1 (\xi) \leq C_1 v(h) \leq C_1 C_2 u_2 (\xi), 
\]
showing that $u_1, u_2$ are equivalent. 
\end{proof}

\subsection{Partitions of unity subordinate to an induced cover}
This subsection is devoted to the explicit construction of a BAPU subordinate to a cover of $\mathcal{O}$ 
induced by $H$.

\begin{definition}\label{def:BAPU}
  Let $\mathcal{Q} = (Q_i)_{i \in I}$ be an admissible cover of
  an open subset $\mathcal{O} \subset \RHat^d$.
  A family ${\Phi = (\varphi_i )_{i \in I}}$ is called a
  \emph{bounded admissible partition of unity} (BAPU) subordinate to $\mathcal{Q}$ if
  \begin{enumerate}[(i)]
    \item  \label{enu:BAPU_Smoothness}
           $\varphi_i \in C_c^{\infty} (\mathcal{O})$ for all $i \in I$;
    \item \label{enu:BAPU_Partition}
          $\sum_{i \in I} \varphi_i (\xi) = 1$ for all
          $\xi \in \mathcal{O}$;
    \item \label{enu:BAPU_Support}
          $\varphi_i (\xi) = 0$ for all $\xi \in \mathcal{O} \setminus Q_i$
          for all $i \in I$;
    \item \label{enu:BAPU_Boundedness}
          $C_{\Phi}
           := \sup_{i \in I} \| \mathcal{F}^{-1} \varphi_i \|_{L^1} < \infty$.
  \end{enumerate}
  A family ${\Phi = (\varphi_i )_{i \in I}}$ satisfying (i)-(iv) is also called a \emph{$\mathcal{Q}$-BAPU}.
\end{definition}

The construction of the BAPU in Theorem \ref{thm:Q-BAPU} below is 
similar in spirit to the construction in \cite[Section 6]{fuehr2015wavelet} 
and relies on the following auxiliary lemma. 

\begin{lemma}[\cite{fuehr2015wavelet}] \label{lem:construction_BAPU}
For $\psi \in \mathcal{S}_{\mathcal{O}} (\mathbb{R}^d)$ and a measurable, relatively compact set $U \subseteq H$, 
the function 
\begin{align} \label{eq:phiU}
\varphi_U : \RHat^d \to [0,\infty), \quad \xi \mapsto \int_U \big| \hat{\psi} (h^T \xi ) \big|^2 \; d\mu_H (h)
\end{align}
is well-defined and $\varphi_U \in \mathcal{D} (\mathcal{O})$, with 
 $\supp \varphi_U \subseteq \overline{U}^{-T} \supp \hat{\psi} \subseteq \mathcal{O}$.

The inverse Fourier transform $\mathcal{F}^{-1} \varphi_U$ of $\varphi_U$ is given by
\[
\big( \mathcal{F}^{-1} \varphi_U \big) (x) = \int_U \frac{\mathcal{F}^{-1} \big(|\hat{\psi} |^2 \big) (h^{-1} x)}{|\det (h)|} \; d\mu_H (h)
\]
for all $x \in \mathbb{R}^d$, with $\| \mathcal{F}^{-1} \varphi_U \|_{L^1} \leq \mu_H (H) \|\mathcal{F}^{-1} (|\hat{\psi}|^2) \|_{L^1}$.
\end{lemma}

\begin{theorem} \label{thm:Q-BAPU}
Let $\mathcal{O} = H^T C$ be the essential frequency support of $H \leq \mathrm{GL}(d, \mathbb{R})$. 
Let $(h_i)_{i \in I}$ be well-spread in $H$ and let $U \subseteq H$ be a measurable, relatively compact set such that $H = \bigcup_{i \in I} h_i U$.
\begin{enumerate}[(i)]
\item Given an enumeration $(i_n)_{n \in \mathbb{N}}$ of $I$, define 
$
U_{i_n} := h_{i_n} U \setminus \bigcup_{m=1}^{n-1} h_{i_m} U
$
for $n \in \mathbb{N}$.
Then $(U_i)_{i \in I}$ is a measurable partition of $H$ satisfying $U_i \subseteq h_i U$ for all $i \in I$. 
\item For an admissible analyzing vector $\psi \in \mathcal{S}_{\mathcal{O}} (\mathbb{R}^d)$, 
with $\widehat{\psi}^{-1} (\mathbb{C} \setminus \{0\} ) \supset C$, 
define the open set $Q := U^{-T} \big( \hat{\psi}^{-1} (\mathbb{C} \setminus \{0\}) \big)$. 
Then $Q \subset \mathcal{O}$ is relatively compact with closure $\overline{Q} \subset \mathcal{O}$
and the collection $\mathcal{Q} = (h_i^{-T} Q)_{i \in I}$ forms a covering of $\mathcal{O}$ induced by $H$. 
Moreover, the family $(\varphi_{U_i})_{i \in I}$ of functions
\[ 
\varphi_{U_i} := \int_{U_i} \big| \hat{\psi} (h^T \xi) \big|^2 \; d\mu_H (h) \in \mathcal{D} (\mathcal{O})
\]
forms a BAPU subordinate to $\mathcal{Q}$. 
\end{enumerate}
\end{theorem}
\begin{proof}
Assertion (i) is a standard fact from measure theory. 

(ii) Let $\psi \in \mathcal{S}_{\mathcal{O}} (\mathbb{R}^d)$ be admissible, 
with $W := \widehat{\psi}^{-1} (\mathbb{C} \setminus \{0\} ) \supset C$. 
Define the open set
\[
Q = \bigcup_{h \in U} h^{-T} W \subseteq \mathcal{O}. 
\]
Then $Q$ is relatively compact since
$\overline{Q} \subseteq \overline{U}^{-T} \supp \hat{\psi} \subseteq \mathcal{O}$ is compact. 
For the covering property $\mathcal{O} = \bigcup_{i \in I} h_i^{-T} Q$, 
simply note that
\[
\bigcup_{i \in I} h_i^{-T} Q =
 \bigcup_{i \in I} (h_i U)^{-T} W \supset H^T C = \mathcal{O},
\]
and $\mathcal{O} \supset \bigcup_{i \in I} h_i^{-T} Q$ since $\mathcal{O}$ is $H^T$-invariant. 

It remains to show that $(\varphi_{U_i} )_{i \in I}$ forms a BAPU subordinate to $\mathcal{Q}$. 
For this, note that an application of Lemma \ref{lem:construction_BAPU} 
yields that $\varphi_{U_i} \in \mathcal{D} (\mathcal{O})$ as defined in \eqref{eq:phiU}
with $\varphi_{U_i} \equiv 0$ on $\mathcal{O} \setminus (h_i^{-T} Q)$, and
\[
\| \mathcal{F}^{-1} \varphi_{U_i} \|_{L^1} \leq \mu_H (U_i) \|\mathcal{F}^{-1} (|\hat{\psi}|^2) \|_{L^1} 
\leq \mu_H (\overline{U})  \|\mathcal{F}^{-1} (|\hat{\psi}|^2) \|_{L^1}, 
\]
yielding that $\sup_{i \in I} \| \mathcal{F}^{-1} \varphi_{U_i} \|_{L^1} < \infty$. 
Lastly, using that $\psi \in \mathcal{S}_{\mathcal{O}} (\mathbb{R}^d)$ is admissible 
and $(U_{i_n})_{n \in \mathbb{N}}$ forms a partitition of $H$, 
it follows that
\[
\sum_{i \in I} \varphi_{U_i} (\xi) = \sum_{n = 1}^{\infty} \varphi_{U_{i_n}} (\xi) = 
\sum_{n = 1}^{\infty} \int_{U_{i_n}} \big| \hat{\psi} (h^T \xi) \big|^2 \; d\mu_H (h) =
\int_{H} \big| \hat{\psi} (h^T \xi) \big|^2 \; d\mu_H (h) = 1,
\]
as desired. 
\end{proof}

\subsection{Besov-type decomposition spaces}

Following \cite{triebel1977fourier, stroeckert1979decomposition, voigtlaender_embeddings}, 
we use so-called \emph{Fourier distributions} as a reservoir for the decomposition spaces.

\begin{definition}\label{def:Reservoir}
  Let $\mathcal{O} \subset \RHat^d$ be open and set
  $
    Z (\mathcal{O})
    := \mathcal{F} (C_c^{\infty}  (\mathcal{O})).
  $
  The space $Z (\mathcal{O})$ will be equipped with the unique topology making the
  Fourier transform $\mathcal{F} : C_c^{\infty} (\mathcal{O}) \to Z(\mathcal{O})$
  into a homeomorphism.
  The topological dual space $(Z (\mathcal{O}))'$ of
  $Z (\mathcal{O})$ is denoted by $Z' (\mathcal{O})$.
  The bilinear duality pairing between $Z'(\mathcal{O})$ and $Z(\mathcal{O})$
  is denoted by
  $( \cdot, \cdot ) = Z'(\mathcal{O}) \times Z(\mathcal{O}) \to \CC$. 
\end{definition}

  The \emph{Fourier transform} of an
  $f \in Z'(\mathcal{O})$ is defined by duality as
  \[
    \mathcal{F} : Z'(\mathcal{O}) \to \mathcal{D}'(\mathcal{O}),
    \quad
               f \mapsto  \widehat{f} := f \circ \mathcal{F}.
  \]

\begin{definition}\label{def:DecompositionSpace}
  Let $p, q \in [1, \infty]$. Let $\mathcal{Q}= (Q_i)_{i \in I}$ be an admissible
  covering of an open set $\mathcal{O} \subset \RHat^d$ with
  a subordinate BAPU $(\varphi_i)_{i \in I}$. Let $u = (u_i)_{i \in I}$ be
  $\mathcal{Q}$-moderate.
  For $f \in Z' (\mathcal{O})$, set
  \begin{align} \label{eq:decomp_norm}
    \| f \|_{\mathcal{D} (\mathcal{Q}, L^p, \ell^q_u) }
    := \Big\|
          (
           \| \mathcal{F}^{-1} (\varphi_i \cdot \widehat{f} \, ) \|_{L^p}
          )_{i \in I}
       \Big\|_{\ell^q_u}
    \in [0,\infty] \, ,
  \end{align}
  and define the associated \emph{Besov-type decomposition space}
  $\mathcal{D}(\mathcal{Q},L^p,\ell^q_u)$ as
  \[
    \mathcal{D} (\mathcal{Q}, L^p, \ell_u^q)
    := \bigg\{
          f \in Z' (\mathcal{O})
          \; : \;
          \|f \|_{\mathcal{D} (\mathcal{Q}, L^p, \ell^q_u) } < \infty
       \bigg\}.
  \]
\end{definition}

\begin{remark}
The decomposition spaces $\mathcal{D} (\mathcal{Q}, L^p, \ell^q_u)$ 
form Banach spaces for all $p,q \in [1,\infty]$ and its definition is 
independent of the chosen BAPU, with equivalent norms for different choices. 
For proofs of these and other properties, the reader is referred to  \cite{voigtlaender_embeddings}. 
\end{remark}

The following result, which is specific for the setting considered here, 
is \cite[Corollary 26]{fuehr2015wavelet}. 

\begin{lemma} \label{lem:pull-back_equivnorm}
Suppose that $\mathcal{Q} = (h_i^{-T} Q)_{i \in I}$ and $\mathcal{Q}' = (h_i^{-T} Q')_{i \in I}$ are
coverings of $
\mathcal{O}$ induced by $H$. Then 
\[
\mathcal{D} (\mathcal{Q}, L^p, \ell^q_{u}) = \mathcal{D} (\mathcal{Q}', L^p, \ell^q_{u})
\]
with equivalent norms for every pull-back $u : \mathcal{O} \to \mathbb{R}^+$ 
obtained from a $v_0$-moderate weight $v : H \to \mathbb{R}^+$ 
relative to a locally bounded, submultiplicative weight $v_0 : H \to \mathbb{R}^+$. 
\end{lemma}

Lastly, we mention some density and approximation results that will be 
used in the next section. 
The following notion of a dominated distribution was introduced in \cite{romero_invertibility}.

\begin{definition} 
Let $I$ be a countable index set, and let $u : I \to \RR^+$ be a weight. 
Define
\[
\ell^q_u (I; L^p) := \big\{ F : I \to L^p (\RR^d) \; : \; \big\| (\| F_i \|_{L^p} )_{i \in I} \big\|_{\ell^q_u} < \infty \big\}
\] 
and equip it with the norm $\| F \|_{\ell^q_u (I; L^p)} =  \big\| (\| F_i \|_{L^p} )_{i \in I} \big\|_{\ell^q_u}$. 
Let $\mathcal{Q} = (Q_i)_{i \in I}$ be an admissible covering 
of an open set $\mathcal{O} \subset \RR^d$ with $\mathcal{Q}$-BAPU $(\varphi_i)_{i \in I}$ 
and let $F \in \ell^q_u (I; L^p)$. 

A distribution $f \in Z'(\mathcal{O})$ is called \emph{$F$-dominated} 
if 
\[
|\mathcal{F}^{-1} ( \varphi_i \cdot \widehat{f})| \leq F_i
\]
for all $i \in I$. 
\end{definition}

The following lemma is \cite[Proposition 3.13]{romero_invertibility}. 

 \begin{lemma}[\cite{romero_invertibility}] \label{lem:Density_decomposition}
  Let $\mathcal{Q} = (Q_i)_{i \in I}$ be an admissible covering of an open set
  $\mathcal{O} \subset \RHat^d$ with $\mathcal{Q}$-BAPU $\Phi = (\varphi_i)_{i \in I}$
   and let $u = (u_i)_{i \in I}$ be a $\mathcal{Q}$-moderate weight.
  Then
  \begin{enumerate}[(i)]
    \item The inclusion
               $\mathcal{S}_{\mathcal{O}} (\RR^d) \subset \mathcal{D}(\mathcal{Q}, L^p, \ell_u^q)$
               holds for all $p,q \in [1,\infty]$.
    \item If $p,q \in [1,\infty)$, then $\mathcal{S}_{\mathcal{O}} (\RR^d)$ is norm
                dense in $\mathcal{D}(\mathcal{Q},L^p,\ell_u^q)$.
    \item If $p,q \in [1,\infty]$ and $f \in \mathcal{D}(\mathcal{Q}, L^p, \ell_u^q)$,
    		then there is an $F \in \ell^q_u (I; L^p)$ 
    		and a constant $C_{\Phi, \mathcal{Q}} > 0$ such that
    		\[
    		\| F \|_{\ell^q_u (I; L^p)} 
    		\leq C_{\Phi, \mathcal{Q}} \| f \|_{\mathcal{D}(\mathcal{Q}, L^p, \ell^q_u)} 
    		\]
    		Moreover, there exists a sequence $(f_n)_{n \in \mathbb{N}}$ of
                $F$-dominated functions $f_n \in \mathcal{S}_{\mathcal{O}} (\RR^d)$ such that
                 $f_n \to f$, with convergence in $Z' (\mathcal{O})$.
  \end{enumerate}
\end{lemma}

For the proof of the following Fatou-like property, cf. \cite[Lemma 36]{fuehr2015wavelet}.

\begin{lemma}[{\cite{fuehr2015wavelet}}] \label{lem:fatou}
Let $\mathcal{D}(\mathcal{Q}, L^p, \ell^q_u)$ be a decomposition space. Suppose $(f_n)_{n \in \mathbb{N}}$ is a sequence in $\mathcal{D}(\mathcal{Q}, L^p, \ell^q_u)$ such that $\liminf_{n \to \infty} \| f_n \|_{\mathcal{D} (\mathcal{Q}, L^p, \ell^q_u)} < \infty$ and $f_n \to f \in Z' (\mathcal{O})$, with convergence in $Z' (\mathcal{O})$. 
Then $f \in \mathcal{D}(\mathcal{Q}, L^p, \ell^q_u)$ and $\| f \|_{\mathcal{D}(\mathcal{Q}, L^p, \ell^q_u)} \leq \liminf_{n \to \infty} \|f_n\|_{\mathcal{D}(\mathcal{Q}, L^p, \ell^q_u)}$. 
\end{lemma}

\section{Identification of coorbit and decomposition spaces} \label{sec:identification}
This section is devoted to identifying coorbit spaces associated to integrably admissible dilation groups
with suitable decomposition spaces. To obtain this, it will first be shown that the identity map 
$I : \mathcal{S}_{\mathcal{O}} (\RR^d) \to \mathcal{S}_{\mathcal{O}} (\RR^d)$ is bi-continuous
from the coorbit space into the decomposition space. The general result will then be 
obtained by a suitable density argument. 

Throughout the section, we will only consider admissible weighting functions $v : G \to \RR^+$ of the form
 $(x,h) \mapsto v_0(h)$ for some weight $v_0 : H \to \RR^+$. Clearly, any such weight is admissible
 with $s = 0$ and we will identify $v_0$ with its trivial extension and simply write $v = v_0$. 
 Given any such weight $v$, 
 denote 
 \[ \tilde{v} : H \to \RR^+, \; h \mapsto |\det (h^{-1})|^{\frac{1}{2} - \frac{1}{q}} v(h^{-1}) \]
 and let $u : \mathcal{O} \to \RR^+$ be a pull-back of $\tilde{v}$
 with discretization $u : I \to \RR^+$ given by 
 \[u_i = u(x_i) = u(h^{-T}_i \xi_0) = \tilde{v} (h_i^{-1}) =  |\det (h_i)|^{\frac{1}{2} - \frac{1}{q}} v(h_i). \] 
 Recall that other choices of pull-backs or discretizations yield equivalent weights; 
 see Remark \ref{rem:discrete_weights} and Lemma \ref{lem:pull-back}.  

\subsection{Continuity, Part I}

For estimating the decomposition space norm of an $f \in \mathcal{S}_{\mathcal{O}} (\RR^d)$, 
we will make use of the following localization result \cite[Lemma 34]{fuehr2015wavelet}. 

\begin{lemma}[\cite{fuehr2015wavelet}] \label{lem:localization}
Let $U \subseteq H \leq \mathrm{GL}(d, \mathbb{R})$ be relatively compact and measurable. For $f, \psi \in \Schwartz(\RR^d)$, 
\[
\bigg( \mathcal{F}^{-1} \big( \varphi_U \cdot \widehat{f} \; \big) \bigg) (x) 
= \int_U |\det (h) |^{-3/2} \big(W_{\psi} f (\cdot, h) \ast D_{h^{-T}} \psi\big) (x) \; d\mu_H (h)
\]
for any $x \in \RR^d$, where $D_{h} f := f(h^T \cdot)$.
\end{lemma}

\begin{proposition} \label{prop:continuity1}
Let $\psi \in \mathcal{S}_{\mathcal{O}} (\RR^d)$ be admissible 
and let $(h_i)_{i \in I}$ be well-spread in $H$
with $H = \bigcup_{i \in I} h_i U$ for some relatively compact unit neighborhood $U \subseteq H$. 
Let $Q = U^{-T} \big( \widehat{\psi}^{-1} (\CC \setminus \{0\}) \big)$ and 
let $\mathcal{Q} = (h_i^{-T} Q)_{i \in I}$
be the corresponding induced covering of $\mathcal{O}$. 
Then 
\[
\| f \|_{\mathcal{D} (\mathcal{Q}, L^p, \ell^q_u )} \lesssim \|f\|_{\CoLw}, \quad f \in \mathcal{S}_{\mathcal{O}} (\RR^d)
\]
for all $p,q \in [1,\infty]$
\end{proposition}
\begin{proof}
Throughout the proof, let $p \in [1,\infty]$ and let $(\varphi_{U_i})_{i \in I}$ be a $\mathcal{Q}$-BAPU as guaranteed by Theorem \ref{thm:Q-BAPU}. 
An application of Lemma \ref{lem:localization} and Young's inequality give
\begin{align*}
\bigg\| \mathcal{F}^{-1} \big( \varphi_{U_i} \cdot \widehat{f} \big) \bigg\|_{L^p} 
&\leq \int_{U_i} |\det(h)|^{-\frac{3}{2}} \| W_{\psi} f (\cdot, h) \ast D_{h^{-T}} \psi \|_{L^p} \; d\mu_H (h) \\
&\leq \| \psi \|_{L^1} \int_{U_i} |\det (h)|^{-\frac{1}{2}} \| W_{\psi} f (\cdot, h) \|_{L^p} \; d\mu_H (h). 
\numberthis \label{eq:localization1}
\end{align*}
for fixed $i \in I$. 

\textbf{Step 1.} (Case $1 \leq q < \infty$). 
Assume that $\mu_H (U_i) > 0$. If $q \in [1,\infty)$, then 
using \eqref{eq:localization1}, together with Jensen's inequality, gives
\begin{align*}
&\bigg\| \mathcal{F}^{-1} \big( \varphi_{U_i} \cdot \widehat{f} \big) \bigg\|^q_{L^p} \\
&\leq \bigg( \| \psi \|_{L^1} \mu_H (U_i) \bigg)^q \frac{1}{\mu_H (U_i)} \int_{U_i} \bigg( |\det (h)|^{-\frac{1}{2}} 
\|W_{\psi} f (\cdot, h) \|_{L^p} \bigg)^q \; d\mu_H (h) \\
&\leq \bigg( \| \psi \|_{L^1} (\mu_H (\overline{U}))^{1-\frac{1}{q}} \bigg)^{q} \int_{U_i} \bigg( |\det (h)|^{\frac{1}{q}-\frac{1}{2}} 
\|W_{\psi} f (\cdot, h) \|_{L^p} \bigg)^q \; \frac{d\mu_H (h)}{|\det (h)|}, \numberthis \label{eq:localization2}
\end{align*}
where we also used that $\mu_H (U_i) \leq \mu_H (\overline{U}) < \infty$. 
Clearly, the estimate \eqref{eq:localization2} remains valid if $\mu_H (U_i) = 0$. 

Write $h \in U_i$ as $h = h_i u$ for some $u \in \overline{U}$ and set 
$C_1 := \min_{h \in \overline{U}} |\det (h)|$ and $C_2 := \max_{h \in \overline{U}} |\det (h)|$. 
Then
\begin{align} \label{eq:C3}
\frac{|\det (h)|^{\frac{1}{q}-\frac{1}{2}}}{|\det (h_i)|^{\frac{1}{q} - \frac{1}{2}}} 
\leq \max \bigg\{C_1^{\frac{1}{q} - \frac{1}{2}}, C_2^{\frac{1}{q} - \frac{1}{2}} \bigg\} =: C_3. 
\end{align}
Recall that $u_i = |\det(h_i)|^{\frac{1}{2} - \frac{1}{q}} v(h_i) = v(h_i) \big(|\det (h_i)|^{\frac{1}{q} - \frac{1}{2}}\big)^{-1}$. Using inequalities \eqref{eq:localization2} and \eqref{eq:C3}, 
a direct calculation gives
\begin{align*}
&\sum_{i \in I} \bigg( u_i \big\| \mathcal{F}^{-1} \big( \varphi_{U_i} \cdot \widehat{f} \big) \big\|_{L^p} \bigg)^q \\
&\leq \bigg(\| \psi \|_{L^1} \big(\mu_H (\overline{U})\big)^{1 - \frac{1}{q}}\bigg)^q 
\sum_{i \in I} \bigg(  \int_{U_i} \bigg( \frac{|\det (h)|^{\frac{1}{q} - \frac{1}{2}} }{|\det (h_i)|^{\frac{1}{q} - \frac{1}{2}}} v(h_i) \; \| W_{\psi} f (\cdot, h) \|_{L^p} \bigg)^q \frac{d\mu_H (h)}{|\det (h)|} \bigg) \\
&\leq C_3^q C_v^q \bigg(\| \psi \|_{L^1} \big(\mu_H (\overline{U})\big)^{1 - \frac{1}{q}}\bigg)^q
\sum_{i \in I} \bigg(  \int_{U_i} \bigg( v(h)  \| W_{\psi} f (\cdot, h) \|_{L^p} \bigg)^q \frac{d\mu_H (h)}{|\det (h)|} \bigg) \\
&= C_3^q C_v^q \bigg(\| \psi \|_{L^1} \big(\mu_H (\overline{U})\big)^{1 - \frac{1}{q}}\bigg)^q \| W_{\psi} f \|_{L^{p,q}_v}^q, 
\end{align*}
where $C_v > 0$ is such that $v(h_i) \leq C_v v(h)$. 

\textbf{Step 2.} (Case $q = \infty$.) If $q = \infty$, then \eqref{eq:localization1} gives
\begin{align*}
\sup_{i \in I} \bigg( u_i  \bigg\| \mathcal{F}^{-1} \big( \varphi_{U_i} \cdot \widehat{f} \big) \bigg\|_{L^p} \bigg) 
&\leq \| \psi \|_{L^1} \sup_{i \in I} \int_{U_i} \bigg(\frac{|\det(h)|}{|\det(h_i)|} \bigg)^{-\frac{1}{2}} v (h_i) \| W_{\psi} f (\cdot, h) \|_{L^p} \; d\mu_H (h) \\
&\leq C_1^{-\frac{1}{2}} C_v \| \psi \|_{L^1} \sup_{i \in I} \int_{U_i} v(h) \| W_{\psi} f (\cdot, h) \|_{L^p} \; d\mu_H (h) \\
&\leq C_1^{-\frac{1}{2}} C_v \mu_H (\overline{U}) \| \psi \|_{L^1} \| W_{\psi} f \|_{L^{p,q}_v},
\end{align*} 
which yields the result. 
\end{proof}

\subsection{Continuity, Part II}

We next complement Proposition \ref{prop:continuity1} with a reverse inequality. 

\begin{proposition} \label{prop:continuity2}
Let $\mathcal{Q} = (h_i^{-T} Q)_{i \in I}$ be any cover of $\mathcal{O}$ induced by $H$
admitting a BAPU $\Phi = (\varphi_i)_{i \in I}$ subordinate to $\mathcal{Q}$.  
Then
\[ \| f \|_{\CoLw} \lesssim \|f \|_{\mathcal{D} (\mathcal{Q}, L^p, \ell^q_u)}, \quad f \in \mathcal{S}_{\mathcal{O}} (\mathbb{R}^d)
\]
for all $p,q \in [1,\infty]$,
\end{proposition}
\begin{proof}
Let $(\varphi_i)_{i \in I}$ be a $\mathcal{Q}$-BAPU 
and let $f, \psi \in \mathcal{S}_{\mathcal{O}} (\RR^d)$ with $\psi$ admissible. 
Set $K := \supp \widehat{\psi}$.   
Then an application of Lemma \ref{lem:neighborhoods_bounded} provides a constant $C_1 > 0$ 
such that $\# I_h \leq C_1$ for all $h \in H$,
with 
\[ I_h = I_h(K,Q) := \{ i \in I \; | \; h^{-T} K \cap h_i^{-T} \overline{Q} \neq \emptyset \}. \] 
For fixed $h \in H$, note that 
\[
1 = \sum_{i \in I} \varphi_i (\xi) = \sum_{i \in I_h} \varphi_i(\xi)
\]
 for all $\xi \in h^{-T} K$.  Hence, for fixed $(x,h) \in \mathbb{R}^d \rtimes H$, 
\begin{align*}
W_{\psi} f (x,h) = \langle \widehat{f}, \widehat{\pi (x,h) \psi} \rangle 
= \sum_{i \in I_h} \big \langle \widehat{f} \cdot \varphi_i, \widehat{\pi(x,h) \psi} \big\rangle
= \sum_{i \in I_h} \big \langle \mathcal{F}^{-1} ( \widehat{f} \cdot \varphi_i ), \pi (x,h) \psi \big \rangle
\end{align*}
and
\begin{align*}
\big \langle \mathcal{F}^{-1} ( \widehat{f} \cdot \varphi_i ), \pi (x,h) \psi \big \rangle
&= |\det (h)|^{-\frac{1}{2}} \int_{\mathbb{R}^d} \big( \mathcal{F}^{-1} (\widehat{f} \cdot \varphi_i) \big) (y) 
\overline{D_{h^{-T}} \psi (y-x) } \; d y \\
&= |\det (h)|^{-\frac{1}{2}} \big( \big( \mathcal{F}^{-1} (\widehat{f} \cdot \psi_i) \big) 
\ast \big( D_{h^{-T}} \psi^*) \big) (x),
\end{align*}
where $\psi^* := \overline{\psi(- \cdot)}$.
Therefore, an application of Young's inequality yields
\begin{align*}
\big\| W_{\psi} f (\cdot, h) \big\|_{L^p} 
&\leq |\det (h)|^{-\frac{1}{2}} \sum_{i \in I_h} 
\big\| \big ( \mathcal{F}^{-1} (\widehat{f} \cdot \varphi_i) \ast \big( D_{h^{-T}} \psi^* \big) \big\|_{L^p} \\
&\leq |\det (h) |^{\frac{1}{2}} \| \psi^* \|_{L^1} \sum_{i \in I_h} \big\| \mathcal{F}^{-1} (\widehat{f} \cdot \varphi_i) \big\|_{L^p}. \numberthis \label{eq:W-L^p-estimate}
\end{align*}
We split the remaining part of the proof into three steps:

\textbf{Step 1.} (Auxiliary cover). 
Recall that since $(h_i)_{i \in I}$ is well-spread, there exists a relatively compact, 
measurable set $U \subseteq H$ such that $H = \bigcup_{i \in I} h_i U$. 
Setting $K_2 := \overline{U}^{-T} K \cup \overline{Q}$ therefore yields an induced cover
$\mathcal{Q}' = (Q_i')_{i \in I}$, where $Q_i' := h_i^{-T} K_2$. 
Hence 
$N_{\mathcal{Q}'} := \sup_{i \in I} i^*_{\mathcal{Q}'} < \infty$, 
where $i_{\mathcal{Q}'}^* := \{ \ell \in I \; : \; Q'_{\ell} \cap Q'_i \neq \emptyset\}$, 
since an induced cover is admissible by Proposition \ref{prop:induced_admissible}. 
For $i \in I$,  $h \in h_i U$ and $j \in I_h$, note that 
\[ \emptyset \neq h^{-T} K \cap h_j^{-T} \overline{Q} \subset Q_i' \cap Q_j',\] 
and thus $I_h \subset i^*_{\mathcal{Q}'}$ for any $i \in I$ and $h \in h_i U$.

\textbf{Step 2.} (Case $1\leq q < \infty$). 
Let $q \in [1,\infty[$, and let $i \in I$ and $h \in h_i U$. 
Using the just proven inclusion $I_h \subset i^*_{\mathcal{Q}'}$ in Step 1, 
we use \eqref{eq:W-L^p-estimate} to estimate 
\begin{align*}
&|\det (h)|^{-1} \bigg( v(h) \big \| W_{\psi} f (\cdot, h) \big\|_{L^p} \bigg)^q \\
&\leq |\det (h)|^{-1} \bigg( v(h) |\det(h)|^{\frac{1}{2}} \|\psi^* \|_{L^1} \sum_{\ell \in i^*_{\mathcal{Q}'}} \big\| \mathcal{F}^{-1} (\widehat{f} \cdot \varphi_{\ell} ) \big\|_{L^p} \bigg)^q \\
&\leq \bigg( \| \psi^* \|_{L^1} |\det (h) |^{\frac{1}{2} - \frac{1}{q}} v(h) \bigg)^q N_{\mathcal{Q}'}^q 
 \bigg\| \bigg(\big\| \mathcal{F}^{-1} (\widehat{f} \cdot \varphi_{\ell} ) \big\|^q_{L^p} \bigg)_{\ell \in i^*_{\mathcal{Q}}} \bigg\|_{\ell^{\infty}} \\
&\leq \bigg( N_{\mathcal{Q}'}   \| \psi^* \|_{L^1} |\det(h)|^{\frac{1}{2} - \frac{1}{q}} v(h) \bigg)^q \; \bigg\| \bigg( \big\| \mathcal{F}^{-1} (\widehat{f} \cdot \varphi_{\ell} ) \big\|^q_{L^p} \bigg)_{\ell \in i^*_{\mathcal{Q}}} \bigg\|_{\ell^{1}}. \numberthis \label{eq:deth_estimate}
\end{align*}
Define $C' := \max_{u \in \overline{U}} |\det(u)|^{\frac{1}{2} - \frac{1}{q}} < \infty$, such that
$ |\det(h)|^{\frac{1}{2} - \frac{1}{q}} \leq C' |\det(h_i)|^{\frac{1}{2} - \frac{1}{q}}$.
Moreover, let
 $C_v > 0$ be such that $v(h) \leq C_v v(h_i)$. 
Then 
\[
|\det (h)|^{\frac{1}{2} - \frac{1}{q}} v(h) \leq C_v C' |\det (h_i )|^{\frac{1}{2} - \frac{1}{q}} v(h_i) 
= C_v C' u_i. 
\]
Since $u : I \to \mathbb{R}^+$ is $\mathcal{Q}'$-moderate by Lemma \ref{lem:pull-back}, 
there exists  $C_u > 0$ such that $u_i \leq C_u u_{\ell}$ for all $i \in I$ and $\ell \in i^*_{\mathcal{Q}'}$. 
Therefore, we can estimate \eqref{eq:deth_estimate} by
\begin{align*}
|\det (h)|^{-1} \bigg( v(h) \big \| W_{\psi} f (\cdot, h) \big\|_{L^p} \bigg)^q 
\leq \bigg( C_{u} N_{\mathcal{Q}'} C' C_v \| \psi^* \|_{L^1} \bigg)^q \; \sum_{\ell \in i^*_{\mathcal{Q}'}} \bigg( u_{\ell} \big\| \mathcal{F}^{-1} ( \widehat{f} \cdot \varphi_{\ell} ) \big\|_{L^p} \bigg)^q.
\end{align*}
By setting $C_2 := C_{u} N_{\mathcal{Q}'} C' C_v \| \psi^* \|_{L^1}$, we therefore get
\begin{align*}
\big \| W_{\psi} f \big \|_{L^{p,q}_v}^q 
&\leq \sum_{i \in I} \int_{h_i U} \bigg( v(h) \big\| W_{\psi} f (\cdot, h) \big\|_{L^p} \bigg)^q \frac{d\mu_H (h)}{|\det (h)|} \\
&\leq \mu_H (\overline{U}) C_2^q \sum_{i \in I} \sum_{\ell \in i^*_{\mathcal{Q}'}} \bigg( u_{\ell} \big\| \mathcal{F}^{-1} (\widehat{f} \cdot \varphi_{\ell} ) \big \|_{L^p} \bigg)^q \\
&\leq \mu_H (\overline{U}) C_2^q N_{\mathcal{Q}'} \| f \|_{\mathcal{D}(\mathcal{Q}, L^p, \ell^q_u)}^q,
\end{align*}
where it is used that $\ell \in i^*_{\mathcal{Q}'}$ if, and only if, $i \in \ell^*_{\mathcal{Q}'}$.  

\textbf{Step 3.} (Case $q = \infty$). 
If $q = \infty$, let $h \in H$ be arbitrary and choose $i \in I$ such that $h \in h_i U$. 
Then a direct calculation using \eqref{eq:W-L^p-estimate} gives
\begin{align*}
v(h) \big\| W_{\psi} f (\cdot, h) \big\|_{L^p} 
&\leq C_v \max_{h \in \overline{U}} |\det (h) |^{\frac{1}{2} - \frac{1}{q}} \| \psi^* \|_{L^1} 
\sum_{\ell \in I_h}  u_i \big\| \mathcal{F}^{-1} (\widehat{f} \cdot \varphi_i) \big\|_{L^p} \\
&\leq C_v C_{\mathcal{Q}'} N_{\mathcal{Q}'} \max_{h \in \overline{U}} |\det (h) |^{\frac{1}{2} - \frac{1}{q}} \| \psi^* \|_{L^1} 
\sup_{\ell \in I} \bigg( u_{\ell} \big\| \mathcal{F}^{-1} (\widehat{f} \cdot \varphi_i) \big\|_{L^p} \bigg) \\
&= C_v C_{\mathcal{Q}'} N_{\mathcal{Q}'} \max_{h \in \overline{U}} |\det (h) |^{\frac{1}{2} - \frac{1}{q}} \| \psi^* \|_{L^1} 
\| f \|_{\mathcal{D}(\mathcal{Q}, L^p, \ell^q_u)}. 
\end{align*}
This completes the proof. 
\end{proof}

\subsection{Isomorphism}
This section finishes the proof of the identification of a coorbit space with a suitable
associated decomposition space. For this, we first relate the reservoirs $\SOdual$ and $Z'(\mathcal{O})$ defining 
$\CoSY$ and $\mathcal{D}(\mathcal{Q}, L^p, \ell^q_u)$, respectively. 

\begin{lemma} \label{lem:canonical_identification}
Let $\mathcal{O} \subset \RHat^d$ be an open set. Then the map 
$\Gamma : \SOdual \to Z'(\mathcal{O})$
defined by 
$
\Gamma (f) (\phi) := \langle f, \overline{\phi} \rangle
$
forms a linear homeomorphism. 
\end{lemma}
\begin{proof}
The map $ \mathcal{S}_{\mathcal{O}} (\RR^d) \ni \phi \mapsto \overline{\phi} \in Z(\mathcal{O})$ forms a homeomorphism. 
The claim follows therefore by duality. 
\end{proof}

In the sequel, we will canonically identify $f \in \SOdual$ and $\Gamma (f) \in Z'(\mathcal{O})$ and write, with an abuse of notation, simply $f = \Theta (f)$. 

\begin{theorem} \label{thm:identification}
Let $H \leq \mathrm{GL}(d, \mathbb{R})$ be integrably admissible 
with essential frequency support $\mathcal{O} \subseteq \RHat^d$.
Let $\mathcal{Q} = (h_i^{-T} Q)_{i \in I}$ be any cover of $\mathcal{O}$ induced by $H$
admitting a BAPU $\Phi = (\varphi_i)_{i \in I}$ subordinate to $\mathcal{Q}$.  
Suppose that $\mathbb{R}^d \rtimes H \ni (x,h) \mapsto v(h) \in \mathbb{R}^+$ is an admissible weight
and let $u : \mathcal{O} \to \RR^+$ be a pull-back of $h \mapsto |\det (h^{-1})|^{\frac{1}{2} - \frac{1}{q}} v(h^{-1})$. Then
\[
\| \cdot \|_{\CoLw} \asymp \| \cdot \|_{\mathcal{D}(\mathcal{Q}, L^p, \ell^q_u)}
\]
 for all $p,q \in [1,\infty]$.
In particular, the spaces $\CoLw$ and $\mathcal{D}(\mathcal{Q}, L^p, \ell^q_u)$ are isomorphic 
as Banach spaces.
\end{theorem}
\begin{proof}
Let $f \in \CoSY$. By Corollary \ref{cor:density_coorbit}, there exists a sequence
$(f_n)_{n \in \mathbb{N}}$ of functions $f_n \in \mathcal{S}_{\mathcal{O}} (\RR^d)$ 
satisfying $f_n \to f$ in $\SOdual$ and 
$\| f_n \|_{\CoSY} \leq C \| f \|_{\CoSY}$ for all $n \in \mathbb{N}$. 
Proposition \ref{prop:continuity1} yields a constant $C' > 0$
such that $\| \cdot \|_{\mathcal{D}(\mathcal{Q}, L^p, \ell^q_u)} \leq C' \| \cdot \|_{\CoSY}$
on $\mathcal{S}_{\mathcal{O}} (\RR^d)$. 
Combining this, together with Lemma \ref{lem:fatou}, yields that
\begin{align*}
\| f \|_{\mathcal{D} (\mathcal{Q}, L^p, \ell^q_u)} 
\leq \liminf_{n \to \infty} \| f_n \|_{\mathcal{D} (\mathcal{Q}, L^p, \ell^q_u)} 
\leq C' \liminf_{n \to \infty} \| f_n \|_{\CoSY} 
\leq C C' \| f \|_{\CoSY},
\end{align*}
which gives the embedding $\CoSY \hookrightarrow \mathcal{D} (\mathcal{Q}, L^p, \ell^q_u)$. 

For the reverse, let $f \in \mathcal{D}(\mathcal{Q}, L^p, \ell^q_u)$. Then, by
Lemma \ref{lem:Density_decomposition}, there exists an $F \in \ell^q_u (I; L^p)$ satisfying
$\|F \|_{\ell^q_u (I; L^p)} \leq C_{\Phi, \mathcal{Q}} \| f \|_{\mathcal{D}(\mathcal{Q}, L^p, \ell^q_u)}$. Moreover, there exists a sequence $(f_n)_{n \in \mathbb{N}}$ of $F$-dominated functions $f_n \in \mathcal{S}_{\mathcal{O}} (\RR^d)$ such that $f_n \to f$
in the weak$^*$-topology on $Z'(\mathcal{O})$. Hence
$\lim_{n \to \infty} W_{\psi} f_n (x) = W_{\psi} f (x)$ for all $x \in G$. 
Therefore
\begin{align*}
\| f \|_{\CoSY} &= \lim_{n \to \infty} \| f_n \|_{\CoSY} 
\leq C_0 \lim_{n \to \infty} \| f_n \|_{\mathcal{D}(\mathcal{Q}, L^p, \ell^q_u)} 
\leq C_0  \|F\|_{\ell^q_u(I; L^p)} \\
&\leq C_0 C_{\Phi, \mathcal{Q}} \| f \|_{\mathcal{D} (\mathcal{Q}, L^p, \ell^q_u)},
\end{align*}
where $C_0 > 0$ is provided by Proposition \ref{prop:continuity2}. 
\end{proof}

\section{Examples and applications} \label{sec:examples}

This section provides examples of (classes of) integrably admissible dilation groups
and their associated coorbit spaces. Moreover, several applications of their realization 
as a Besov-type decomposition space  are given. 

\begin{example}
Let $H \leq \mathrm{GL}(d, \RR)$ be (irreducibly) admissible, 
i.e., there exists a single open orbit $\mathcal{O} = H^T \xi_0$, for $\xi_0 \in \RHat^d$, 
of full measure
 for which the stabilizer group $H_{\xi_0}$ is compact. 
Then the quasi-regular representation $(\pi, L^2 (\RR^d))$ of $G = \RR^d \rtimes H$
is a discrete series representation by \cite[Corollary 21]{fuehr2010generalized}. 
Moreover, since $H$ is integrably admissible, the representation $(\pi, L^2 (\RR^d))$
is $w$-integrable for any admissible control weight $w : G \to \mathbb{R}^+$ by Proposition \ref{prop:amalgam_estimate}. 
Thus the coorbit space theory developed in \cite{feichtinger1989banach1, feichtinger1989banach2}
is therefore applicable in this setting. We will relate the original spaces \cite{feichtinger1989banach1, feichtinger1989banach2} to the ones defined in \eqref{eq:definition_coorbit}. 

Let $v : G \to \mathbb{R}^+$ be an admissible weight and let $w : G \to \mathbb{R}^+$ be an admissible control weight of the same type. 
Fix a vector $\psi \in L^2 (\RR^d) \setminus \{0\}$ satisfying $W_{\psi} \psi \in L^1_w (G)$ and 
define 
 \[ \mathcal{H}_{1,w} = \bigg\{ f \in L^2 (\RR^d) \; : \; W_{\psi} f \in L^1_w (G) \bigg\} \neq \emptyset. \]
Denoting by $\overline{\mathcal{H}'_{1,w}}$ the anti-dual space of $\mathcal{H}_{1,w}$, 
the coorbit spaces of \cite{feichtinger1989banach1} are defined in terms of $\overline{\mathcal{H}'_{1,w}}$ as
\begin{align} \label{eq:coorbit_original}
\CoHLw := \bigg\{ f \in \overline{\mathcal{H}'_{1,w}} \; : \; 
 \langle f, \pi(\cdot) \psi \rangle \in L^{p,q}_v (G) \bigg\}. 
\end{align}
Note that any non-zero $\psi \in \mathcal{S}_{\mathcal{O}} (\RR^d)$ can be used to define the space
$\mathcal{H}_{1,w}$. 

By \cite[Corollary 11]{fuehr2015wavelet},
the mapping 
$ 
\Lambda : \overline{\mathcal{H}'_{1,w}} \to Z' (\mathcal{O}), \; f \mapsto \big( \phi \mapsto f (\overline{\phi}) \big)
$
is well-defined, injective, linear and continuous with respect to the weak$^*$-topology on $\overline{\mathcal{H}'_{1,w}}$. 
Using the map $\Gamma : \SOdual \to Z'(\mathcal{O})$ 
of Lemma \ref{lem:canonical_identification}, it follows by \cite[Theorem 38]{fuehr2015wavelet}
that $\Gamma^{-1} \circ \Lambda : \overline{\mathcal{H}'_{1,w}} \to \SOdual$ induces an isometric isomorphism 
$
\Gamma^{-1} \circ \Lambda : \CoHLw \to \CoLw.
$
Thus, up to suitable identification, for all $p,q \in [1,\infty]$, 
\[
\CoHLw = \CoLw,
\]
provided that both spaces are defined in terms of some admissible 
$\psi \in \mathcal{S}_{\mathcal{O}} (\RR^d)$. 

For applications of the realization of a coorbit space as a decomposition space in 
the setting of irredubily admissible dilation groups, the reader is referred to \cite[Section 9]{fuehr2015wavelet} 
and \cite{koch_thesis}. 
\end{example}

The following example treats the anisotropic Besov spaces 
considered in \cite{bownik2005atomic, barrios2011characterizations}. 

\begin{example}
Let $H = \langle A \rangle = \{A^j \; : \; j \in \mathbb{Z} \}$ be the cyclic group generated by $A \in \mathrm{GL}(d, \RR)$. 
Recall that a matrix $A \in \mathrm{GL}(d, \RR)$ is called \emph{expansive} if all its
eigenvalues $\lambda \in \sigma (A)$ satisfy $|\lambda|>1$. 
The dilation group $H = \langle A \rangle$ is integrably admissible if,
and only if, either $A$ or its inverse $A^{-1}$ is expansive, 
see \cite{larson2006explicit, schulz2004projections}. 
In this case, the essential frequency support of the quasi-regular representation $(\pi, L^2 (\RR^d))$ 
of $G = \RR^d \rtimes H$ is given by $\mathcal{O} = \RHat^d \setminus \{0\}$. 

Given an open set $C \subset \RHat^d$ such that $\overline{C} \subset \RHat^d \setminus \{0\}$
is compact, a cover $\mathcal{Q}_A = (Q_j)_{j \in \mathbb{Z}}$ of $\RHat^d \setminus \{0\}$, with
$Q_j := A^j \overline{C}$, is called an \emph{homogeneous cover induced by A}, 
see \cite{fuehr_classification}. 
A homogeneous cover induced by $A^T$ is readily seen to be an induced cover of 
$\mathcal{O}$ in the sense of Definition 
\ref{def:induced_cover}. 

For $\alpha \in \mathbb{Z}$, define the weighting function $u_{\alpha, A} : \mathbb{Z} \to \RR^+$ 
by $u_{\alpha, A} (j) = |\det (A) |^{j \alpha}$. 
Then, by \cite[Theorem 5.6]{fuehr_classification}, 
for $p, q \in [1,\infty]$, 
the decomposition space $\mathcal{D} (\mathcal{Q}_{A^T}, L^p, \ell^q_{v_{\alpha, A}})$ 
 coincides, up to suitable identification, with the (homogeneous) anisotropic Besov space
\[
\dot{B}^{p,q}_{\alpha} (\RR^d; A) := \bigg\{ f \in \mathcal{S}' (\RR^d) \; : \; 
\big\| \big( \| f \ast \varphi_j \|_{L^p} )_{j \in \mathbb{Z}} \big\|_{\ell^q_{u_{\alpha, A}}} < \infty \bigg\},
\]
where $\varphi_j := |\det (A)|^j \varphi (A^j \cdot)$ for some $\varphi \in \mathcal{S} (\RR^d)$ 
satisfying $\supp \widehat{\varphi} \subset [-1,1]^d \setminus \{0\}$ and
$\sum_{j \in \mathbb{Z}} \big|\widehat{\varphi} ((A^T)^j \xi ) \big| > 0$ for all $\xi \in \RHat^d \setminus \{0\}$. 
Thus, by means of Theorem \ref{thm:identification}, the anisotropic Besov space 
$\dot{B}^{p,q}_{\alpha} (\RR^d; A)$ can be canonically identified
as a coorbit space $\Co(L^{p,q}_v (\RR^d \rtimes \langle A \rangle))$.

We will next detail the classification results
obtained in \cite{fuehr_classification} to state similar results for the coorbit spaces. 
For this, define an \emph{$A$-homogeneous quasi-norm}  as a Borel map
$\rho_{A} : \RHat^d \to [0, \infty)$ that is positive definite, $A$-homogeneous, i.e., 
$\rho_A (A x) = |\det (A)| \rho_A (x)$, and satisfies  the quasi-triangle inequality
$\rho_A (x+y) \lesssim (\rho_A (x) + \rho_A (y))$. As usual, two quasi-norms
$\rho_A, \rho_B$ are called \emph{equivalent} if $\rho_{A} \asymp \rho_{B}$. 
In this terminology, \cite[Theorem 5.10]{fuehr_classification} asserts
that, given two expansive matrices $A_1, A_2 \in \mathrm{GL}(d, \mathbb{R})$, 
it holds that
$
\dot{B}^{p,q}_{\alpha} (\RR^d; A_1) = \dot{B}^{p,q}_{\alpha} (\RR^d; A_2)
$
for all $p, q \in [1,\infty]$ if, and only if, 
$\rho_{A_1^T}$ and $\rho_{A_2^T}$
are equivalent. As a direct consequence, the corresponding coorbit spaces
coincide, i.e., 
\[
\Co(L^{p,q}_v (\RR^d \rtimes \langle A_1 \rangle))
= \Co(L^{p,q}_v (\RR^d \rtimes \langle A_2 \rangle))
\]
for all $p, q \in [1,\infty]$ if, and only if, the equivalence
$\rho_{A_1^T} \asymp \rho_{A_2^T}$ holds. 

For explicit and checkable criteria for the equivalence of homogeneous quasi-norms, 
the interested reader is referred to \cite[Section 10]{bownik2003anisotropic} and \cite[Section 6]{fuehr_classification}, with the caveat that the latter source corrects some fallacies contained in the earlier one. 
\end{example}

We next  comment on coorbit spaces associated to one-parameter subgroups in more detail. 
The following result formulates a criterion for the integrable admissibility of these groups, 
it is \cite[Theorem 1.1]{grochenig1992compact} rephrased in the terminology used in this paper. 

\begin{proposition} \label{prop:one-parameter}
 Let $A \in \mathbb{R}^{d \times d}$. The associated one-parameter group $H = \exp(\mathbb{R} A)$ is integrably admissible if, and only if, either the real parts of all eigenvalues are strictly positive or strictly negative. The essential frequency support associated to $H$ is given by $\mathcal{O} = \mathbb{R}^d \setminus \{ 0 \}$. 
\end{proposition}

The condition on $A \in \mathbb{R}^{d \times d}$ formulated in the proposition is equivalent to saying that the eigenvalues of $\exp(A)$ either all have modulus $< 1$ or all have modulus $>1$, which is precisely the condition of {\em expansiveness} underlying the construction of anisotropic Besov spaces according to \cite{bownik2003anisotropic}. This raises the very natural question how the coorbit spaces associated to $\exp(\mathbb{R} A)$ are related to those associated to $\langle \exp(A) \rangle$, i.e., to (homogeneous) anisotropic Besov spaces associated to $\exp(A)$. Since $\langle \exp(A) \rangle \subset \exp(\mathbb{R} A)$, with compact quotient, this question is a special case of the following somewhat more general observation. 

\begin{lemma} \label{lem:cocompact}
 Let $H_1 \leq H_2 \leq GL(d,\mathbb{R})$ be two closed subgroups such that $H_2/H_1$ compact. 
 Suppose that $H_2$ is integrably admissible with essential frequency support $\mathcal{O} \subset \RHat^d$. 
 Then the following assertions hold:
 \begin{enumerate}[(i)]
 \item The group  $H_1 \leq \mathrm{GL}(d, \mathbb{R})$ is integrably admissible with essential frequency support $\mathcal{O}$. 
 \item The groups $H_1, H_2 \leq \mathrm{GL}(d, \mathbb{R})$ induce the same coorbit spaces in the following sense:
 
   If
  $v_1, v_2$ are admissible weights on $\mathbb{R}^d \rtimes H_1$ respectively $\mathbb{R}^d \rtimes H_2$,
  depending only on the second variable,  
  such that $v_1$ is equivalent to the restriction $v_2|_{\mathbb{R}^d \rtimes H_1}$, then
 \[ \Co(L^{p,q}_{v_1}(\mathbb{R}^d \rtimes H_1 )) = \Co(L^{p,q}_{v_2}(\mathbb{R}^d \rtimes H_2)) \] 
 for all $1 \le p,q \le \infty$. 
 \end{enumerate}
\end{lemma}

\begin{proof}
If $H_2/H_1$ is compact, then there exists a compact set $K \subset H_2$ with $KH_1 = H_2$. Then, if $C \subset \mathcal{O}$ is a relatively compact open subset with $H_2 C = \mathcal{O}$ and 
$((C,C)) \subset H_2$ compact, it is easy to see that $KC \subset \mathcal{O}$ is relatively compact and open as well, with $((KC,KC)) \subset H_1$ compact, and $H_1 K C = H_2 C = \mathcal{O}$. 
This proves (i). 

In order to show (ii), we make use of their decomposition space realization.  
Furthermore, if $(h_i)_{i \in I} \subset H_1$ is well-spread, it is also a well-spread subset of $H_2$.

By assumption on $v_1$ and $v_2$, transferring either $v_1$ or $v_2$ to the induced covering results in equivalent weights on the covering, and therefore the induced decomposition spaces coincide. Theorem \ref{thm:identification} now yields the desired conclusion. 
\end{proof}

\begin{remark}
In the setting of the previous lemma, 
it is easy to see that restricting an admissible weight from $G_2$ to $G_1$ yields an admissible weight on the smaller group. It is less clear whether every admissible weight on the smaller group extends (up to equivalence) to the larger one. For weights only depending on the $H_i$-variable, some partial answers are available. For instance, if $H_2$ is a abelian, or a direct product of $H_1$ and a second (necessarily compact) subgroup, such an extension is easily constructed. In this case, since $H_2/H_1$ is assumed compact, there exists a measurable cross-section $\sigma: H_2/H_1 \to H_2$ with relatively compact image and with $\sigma(h) = e$ for all $h \in H_1$. Here one checks immediately  $v_2(h) = v_1(\sigma(h)^{-1} h)$ yields an extension of $v_1$ to $H_2$ with the desired properties.  
\end{remark}

\begin{example} \label{ex:threegroups}
 Consider the group $H$ described in Example \ref{ex:integrablyadmissible}(\ref{ex:two-parameter}), one has 
 $H_1 \subset H \subset H_3$, where $H_3$ is the diagonal group, and $H_1 = \exp(\mathbb{R} C)$ with 
 \[
  C= \left[ \begin{array}{ccc} 1 && \\ & 1&\\&&\alpha + \beta\end{array} \right]~.
 \] Hence, $H_1, H_3$ are both integrably admissible. 
 However, $H/H_1 \cong \mathbb{R} \cong H_3/H$, hence Lemma \ref{lem:cocompact} is not applicable,
 and we expect that the coorbit spaces associated to $H$ are different from those associated to either $H_1$ or $H_3$. Moreover, unlike for co-compactly contained pairs of groups, the essential frequency supports of $H_1, H_3$ on the one hand, and of $H$ on the other, are different. 
\end{example}

We expect that a more direct proof of Lemma \ref{lem:cocompact}, relating integration over $H_2$ to integration over $H_1$ by Weil's integral formula, is also available.  However, the proof using the decomposition space identification is remarkably effortless. Note also that already the formulation of the lemma is greatly facilitated by the choice of a common reservoir. In the absence of a common reservoir,  comparing coorbit spaces associated to different groups requires making potentially cumbersome identifications.

\begin{corollary} \label{cor:coorbit_besov}
Let $H = \exp(\mathbb{R} A)$ denote an integrably admissible one-parameter group. Then the coorbit spaces induced by $G = \mathbb{R}^d \rtimes H$ are precisely the anisotropic Besov spaces associated to $\exp(A)$.  
\end{corollary}

\begin{remark}
There is a converse to Corollary \ref{cor:coorbit_besov}. Given any expansive matrix $A$, there exists a second expansive matrix $B$ with the property that $A$ and $B$ induce the same scales of anisotropic Besov spaces, and in addition, $B = \exp(C)$ for a suitable matrix $C$; see \cite[Lemma 7.8]{fuehr_classification}. Thus the classes of anisotropic Besov spaces on the one hand, and of coorbit spaces associated to one parameter groups on the other, coincide precisely. 
\end{remark}

The following result is well-known, and was already stated in \cite{feichtinger1988unified}. 
\begin{corollary}
 Let $H = \mathbb{R}^+\cdot SO(d)$. Then the coorbit spaces of $H$ and $ \mathbb{R}^+\cdot {\rm Id}_{\mathbb{R}^d}$ coincide. 
\end{corollary}

 We sketch a further potential benefit of the ability to switch groups in the description of coorbit spaces. One natural problem arising in the study of coorbit spaces is the question which dilation matrices $g \in {\rm GL}(d,\mathbb{R})$ leave the coorbit spaces invariant. It is shown in \cite{fuehr2015coorbit} that this set of matrices may well depend on the choice of underlying dilation group $H$, but a deeper understanding of this dependence is currently missing. By construction of the coorbit spaces, it is clear that the elements of $H$ leave coorbit spaces associated to $H$ invariant, and more generally, natural candidates are the elements of the {\em normalizer} 
 $N(H)$ of $H$ defined by
 \[
  N(H) = \{ g \in GL(d,\mathbb{R}) : g^{-1} H g = H \}.
 \] 
The following is shown by straigthforward calculations. 
 
 \begin{lemma}
 Let $H \leq \mathrm{GL}(d, \mathbb{R})$ be integrably admissible 
 and let $v : G \rtimes H \to \mathbb{R}^+$ be an admissible weight, depending only on the second variable.
 Then, for any $g \in N(H)$, the action
 \[ \pi(0, g)  : \Co(L^{p,q}_v(\mathbb{R}^d \rtimes H)) 
  \to \Co(L^{p,q}_{v_g} (\mathbb{R}^d \rtimes H)) \] 
  is an isomorphism, where $v_g(h) := v( ghg^{-1})$. 
\end{lemma}

Once the role of the normalizer is understood, a benefit of Lemma \ref{lem:cocompact} becomes apparent, that is best exemplified with the help of the groups $H_1 = \mathbb{R}^+ \cdot {\rm Id}_{\mathbb{R}^d}$ and $H_2 = \mathbb{R}^+ \cdot SO(d)$. Trying to use the normalizer of $H_2$ to find additional dilational symmetries of the associated coorbit spaces yields very little, since
$N(H_2) $ turns out to be a finite extension of $H_2$. 
By contrast, since $H_1$ is central in $GL(d,\mathbb{R})$, one has $N(H_1) = GL(d,\mathbb{R})$, and in addition, $v_g = v$ for every weight on $H_1$ and every $g \in GL(d,\mathbb{R})$. Thus computing the normalizer of $H_1$ is sufficient to see that the isotropic homogeneous Besov spaces are invariant under arbitrary dilations. 

\section*{Conclusion and outlook}
The decomposition space description of coorbit spaces can be understood as a generalization of the $\varphi$-transform characterization of Besov spaces due to Frazier and Jawerth \cite{Frazier1985Besov, Frazier1990Triebel}. As outlined in \cite{fuehr2015wavelet}, the chief benefit of this description comes from the fact that it provides a natural common framework for the unified treatment of spaces associated to different groups, and the papers \cite{voigtlaender_embeddings,voigtlaender_Sobolev} provide tools to tackle the question of embeddings between decomposition spaces, but also of decomposition spaces into Sobolev and BV spaces, in a systematic manner. As one keeps introducing new classes of function spaces, one fundamental question becomes particularly relevant, namely that of {\em classification}, e.g., the problem of deciding whether two different dilation groups actually induce different scales of coorbit spaces. On the decomposition space level, 
this question is settled in \cite{voigtlaender_embeddings}. This result was used for the classification of anisotropic Besov spaces in \cite{fuehr_classification}, and of shearlet coorbit spaces in \cite{koch_thesis}. Furthermore, the latter develops an interesting connection between decomposition space theory and  coarse geometry \cite{MR2007488}, which will be further pursued in upcoming publications. In this paper, a first glimpse of coarse geometric reasoning can be seen in the discussion in Section \ref{sec:examples}, in particular in connection with Lemma \ref{lem:cocompact} and its applications. We expect that the question whether a converse to this lemma holds, e.g. as described in Example \ref{ex:threegroups}, can be ultimately determined using coarse geometry methods. Likewise, it seems plausible to us that the question whether the essential frequency support of an integrably admissible dilation group is unique is ultimately a coarse geometric one. 

It should also be emphasized that, while the coorbit theory developed here is ultimately a branch of decomposition space theory, the focus on this class allows the formulation of sharper results. This can be exemplified by the discretization results such as Theorem \ref{thm:atomic}, which clearly benefit from the underlying group. Further cases in point are provided by the discussion of dilational symmetries in Section \ref{sec:examples}, and by the
classification results in \cite{koch_thesis}, which also make crucial use of the group structure underlying the construction of the coorbit spaces.

An interesting question for future work regards explicit criteria and existence proofs for atomic decompositions using vectors that are compactly supported in space rather than in frequency. For the irreducible case, these were obtained in \cite{fuehr2016vanishing,MR3683671}. The techniques developed for proving the existence of such wavelets in the irreducible case are already fairly involved, and they make heavy use of the fact that the dual action has a unique open orbit. It is therefore expected that the extension to integrably admissible dilation groups 
is highly non-trivial.

\end{document}